\documentclass[10pt,a4paper,onefignum,oneeqnum,onetabnum]{siamart220329}

\usepackage{amsmath,amsfonts,amssymb,amscd,mathrsfs,array,subeqnarray}
\usepackage[markup=underlined]{changes}
\usepackage{todonotes}
\setcommentmarkup{\todo[color={authorcolor!20},size=\scriptsize]{#3: #1}}

\newtheorem{example}{\bf Example}[section]
\newtheorem{remark}[theorem]{Remark}
\allowdisplaybreaks

\usepackage{geometry}
\begin{document}
\date{}

\title{Stochastic linear-quadratic differential game with Markovian jumps in an infinite horizon\thanks{Fan Wu and Xin Zhang are supported by the National Natural Science Foundation of China (Grant Nos. 12171086, 12371472) and by Fundamental Research Funds for the Central Universities grant 2242024K40018. Xun Li is supported by RGC of Hong Kong (Grant Nos. 15216720, 15221621, 15226922) and partially from PolyU 1-ZVXA. Jie Xiong is supported by National Key R\&D Program of China (Grant No. 2022YFA1006102), and the National Natural Science Foundation of China (Grant No. 12326368).}}
\author{Fan Wu \thanks{School of Mathematics, Southeast University, Nanjing 211189, China}\and
Xun Li \thanks{Department of Applied Mathematics, The Hong Kong Polytechnic University, Hong Kong, China; Email: li.xun@polyu.edu.hk}\and Jie Xiong \thanks{Department of Mathematics and SUSTech International Center for Mathematics, Southern University of Science and Technology, Shenzhen, Guangdong, 518055, China; E-mail: xiongj@sustech.edu.cn}\and
Xin Zhang \thanks{Corresponding Author: School of Mathematics, Southeast University, Nanjing 211189, China; E-mail: x.zhang.seu@gmail.com}} \maketitle

\noindent{\bf Abstract:}
This paper investigates a two-person non-homogeneous linear-quadratic stochastic differential game (LQ-SDG, for short) in an infinite horizon for a system regulated by a time-invariant Markov chain. Both non-zero-sum and zero-sum LQ-SDG problems are studied. It is shown that the zero-sum LQ-SDG problem can be considered a special non-zero-sum LQ-SDG problem. The open-loop Nash equilibrium point of non-zero-sum (zero-sum, respectively) LQ-SDG problem is characterized by the solvability of a system of constrained forward-backward stochastic differential equations (FBSDEs, for short) in an infinite horizon and the convexity (convexity-concavity, respectively) of the performance functional and their corresponding closed-loop Nash equilibrium strategy are characterized by the solvability of a system of constrained coupled algebra Riccati equations (CAREs, for short) with certain stabilizing conditions. In addition, the closed-loop representation of open-loop Nash equilibrium point for non-zero-sum (zero-sum, respectively) LQ-SDG is provided by the non-symmetric (symmetric, respectively) solution to a system CAREs. At the end of this paper, we provide three concrete examples and solve their open-loop/closed-loop Nash equilibrium strategy based on the obtained results.

\medskip

\medskip


\section{Introduction}\label{section-introduction}
The game theory plays a vital role in many fields, and its earlier studies can be traced back to \cite{Zermelo-1912-anwendung,Von-1928-theorie,Von-1947-theory}. A significant milestone in game theory is Nash's definition of the Nash equilibrium point in the 1950s and his research on non-cooperative games \cite{Nash-1951-non}. Ba{\c{s}}er and Olsder \cite{Basar-1999-dynamic} systematically introduced the dynamic non-cooperative game theory, and Issacs \cite{Isaacs-1965-differential} originally studied the deterministic differential game. In the past few years, the SDG theory has received intensive attention. Many researchers are especially concerned about the LQ-SDG problem due to its structure and properties. We refer the interested readers to \cite{Hamadene-1998-backward,Hamadene-1999-nonzero,Yu-2012-LQG,Li-2015-Recursive,Sun_linear_2014,Li_linear-quadratic_2021,Wang_time-inconsistent_2022,Wang_kind_2018,Wu_linear_2005,Zhang_backward_2011} for the studies of LQ-SDG problem in the finite horizon and to \cite{Zhu-Zhang-Bin-2014,Sun.JR_2016_IZSLQI,Li-Shi-Yong-2021-ID-MFLQ-IF} for the studies investigated the LQ-SDG problem in the infinite horizon. 

On the other hand, studies on the stochastic model of jump linear systems can be traced back to the work of Krasosvkii and Lidskii \cite{Krasovskii-1961-analytical}. Particularly, the dynamic system involving regime-switching jumps can be used to characterize the instantaneous changes of system parameters and is widely used in various fields, such as engineering, financial management, and economics.  See, for example, Li et al. \cite{Li-Zhou-Rami-2003-ID-MLQ-IF}, Ji and Chizeck \cite{Ji-Chizeck-1991-D-MLQG-F,Ji-Chizeck-1990-D-MLQ-I/F}, Zhang et al. \cite{zhang2021open,zhang_general_2018,zhang_stochastic_2012,zhang-Elliott-Siu-Guo-2011,Zhang-Siu-Meng-2010}, Wu et al. \cite{Wu-Tang-Meng-2023}, Sun et al. \cite{sun_risk-sensitive_2018}, Wen et al. \cite{Wen_2023} and the references therein.

Now, let us mention a few recent literature related to our present paper. Zhang et al. \cite{zhang2021open} investigated a stochastic LQ problem of a Markovian regime-switching system in the finite horizon. It has shown that the open-loop solvability is equivalent to the existence of an adapted solution to the corresponding optimality system with constraint, the closed-loop solvability is equivalent to the existence of a regular solution to Riccati equations, and the strongly regular solvability of Riccati equations is equivalent to the uniform convexity of the cost functional. In addition, they also provided an example to show that open-loop solvability can not imply closed-loop solvability. Wu et al. \cite{Wu-etal} generalized their result to the infinite horizon case and showed the equivalence between the open-loop solvability and closed-loop solvability, which are all equivalent to the existence of a static stabilizing solution to the associated constrained CAREs. Note that in \cite{Wu-etal}, some notions of $L^{2}$-stabilizability for linear stochastic differential system regulated by a Markov chain are introduced, which are interestingly different from the classical mean-square stabilizability studied by \cite{Rami-Zhou-Moore-2000-ID-LQ-IF,Rami-Zhou-2000-LMI-RE-IDLQIF} and can be considered as an extension of Huang et al. \cite{Jianhui-Huang-2015}. Recently, Sun et al. \cite{Sun.JR_2016_IZSLQI} studied a two-person zero-sum LQ-SDG problem in the infinite horizon. It characterized the closed-loop saddle points by the solvability of an algebraic Riccati equation with a certain stabilizing condition. Li et al. \cite{Li-Shi-Yong-2021-ID-MFLQ-IF} successfully generalized their study to the control system with mean-field and treated the corresponding LQ problem, non-zero-sum LQ-SDG problem and zero-sum LQ-SDG problem in a unified framework. Zhu et al. \cite{Zhu-Zhang-Bin-2014} investigated a class of non-zero-sum LQ-SDG problems with Markovian jumps under the mean-square stabilizability framework. However, the diffusion term of state process in \cite{Zhu-Zhang-Bin-2014} does not include control. 

This paper will further discuss the LQ Nash differential game with Markovian jumps based on our earlier work \cite{Wu-etal}. Different from \cite{Zhu-Zhang-Bin-2014}, the control can be entered in diffusion term of state process in our model, and both non-zero-sum and zero-sum LQ-SDG problems are studied successively. Inspired by Li et al. \cite{Li-Shi-Yong-2021-ID-MFLQ-IF}, we consider the zero-sum LQ-SDG problem as a special non-zero-sum LQ-SDG problem and solve the zero-sum LQ-SDG problem based on the results of non-zero-sum LQ-SDG problem. Compared with the existing literature,  the main contributions of this paper can be concluded as follows:
\begin{enumerate}
    \item We expand the results in \cite{Sun.JR_2016_IZSLQI} to the regime-switching jump-diffusion model and investigate not only zero-sum LQ-SDG problem but also non-zero-sum LQ-SDG problem. Our study also generalized the results in \cite{Zhu-Zhang-Bin-2014} by allowing the control to enter the diffusion term and discussing such a problem under the  $L^{2}$-stabilizability framework rather than the mean-square stabilizability framework.
    \item Compared with Li et al. \cite{Li-Shi-Yong-2021-ID-MFLQ-IF}, we use a more direct method to study the open-loop and closed-loop solvability of non-zero-sum LQ-SDG problem. We split non-zero-sum  LQ-SDG problem into two coupled LQ control problems, whose open-loop solvability and closed-loop solvability have been solved in our earlier work \cite{Wu-etal}. And then, based on the results in \cite{Wu-etal}, we obtain the open-loop solvability and closed-loop solvability of non-zero-sum LQ-SDG problem directly. In addition, we also verify that the results based on such a direct method are consistent with those in Li et al. \cite{Li-Shi-Yong-2021-ID-MFLQ-IF}, which consider a mean-field non-zero-sum LQ-SDG problem.
    \item Both our study and  Li et al. \cite{Li-Shi-Yong-2021-ID-MFLQ-IF} characterized the open-loop Nash equilibrium point of non-zero-sum (zero-sum, respectively) LQ-SDG  as the solvability of a system of constrained FBSDEs and the convexity (convexity-concavity, respectively) of the performance functional. However, Li et al. \cite{Li-Shi-Yong-2021-ID-MFLQ-IF} does not provide sufficient conditions to verify the convexity and convexity-concavity of the corresponding performance functional. Hence, they can only construct the closed-loop representation for the open-loop Nash equilibrium point in their numerical section, assuming that the performance functional satisfies the corresponding convexity/concavity condition. In this paper, we have improved this point and provide not only a sufficient condition to verify the convexity of cost functionals for non-zero-sum LQ-SDG problem but also a sufficient condition to verify the convexity-concavity of performance functional for zero-sum LQ-SDG problem (see Remark \ref{rmk-GLQ-open-FBSDEs-1} and Proposition \ref{prop-convex-concave}). Such results can be used to verify that performance functionals in numerical examples satisfy the corresponding convexity/concavity condition.
    \item In the numerical section, we allow both the state and control processes to be multi-dimensional. Using linear matrix inequality (LMI, for short) and semidefinite programming (SDP, for short) techniques (see, for example, \cite{Boyd-Ghaoui-Feron-Balakrishnan-1994-LMI,Vandenberghe-Boyd-1996-SDP}), we solve the corresponding CAREs in each instance, and use them to construct closed-loop representation strategy and closed-loop Nash equilibrium strategy of each example. Compared with \cite{Sun.JR_2016_IZSLQI} that, in the numerical section, discussed the corresponding zero-sum LQ-SDG in a one-dimensional case, our work is more challenging.
\end{enumerate}

The rest of the paper is organized as follows. In Section \ref{section-formulation}, we formulate the problem under the $L^{2}$-stabilizable framework, and both non-zero-sum and zero-sum LQ-SDG problems are introduced.
Section \ref{section-GLQ} aims to give results on the non-zero-sum LQ-SDG problem, including the open-loop solvability, their closed-loop representation, and the closed-loop solvability with constrained CAREs.
Section \ref{section-ZLQ} investigates the open-loop and closed-loop saddle points for zero-sum SLQ differential games. Some examples are presented in Section \ref{section-examples} to illustrate the results developed in the earlier sections.


\section{Problem formulation}\label{section-formulation}
 Let $(\Omega,\mathcal{F},\mathbb{F},\mathbb{P} )$ be a complete filtered probability space on which a one-dimensional standard Brownian motion $W$ and a continuous time, irreducible Markov chain $\alpha$ are defined with $\mathbb{F}=\{\mathcal{F}_{t}\}_{t\geq 0}$ being its natural filtration augmented by all $\mathbb{P}$-null sets in $\mathcal{F}$. We denote the state space of Markov chain $\alpha$ as $\mathcal{S}:=\left\{1,2,...,L\right\}$ whose generators are given by $\Pi=\left\{\pi_{ij}\mid i,j\in\mathcal{S}\right\}$, where $\pi_{ij}\geq 0$ for $i\neq j $ and $\pi_{ii}=-\sum_{j\neq i}\pi_{ij}$. Let $N_{j}(t)$ be the number of jumps to state $j$ up to time $t$. Then for any $j\in\mathcal{S}$, $\widetilde{N}_{j}(t)\triangleq N_{j}(t)-\int_{0}^{t}\sum_{i\neq j}\left[\pi_{ij}\mathbb{I}_{\{\alpha_{s-}=i\}}(s)\right]ds, $
is a $\left(\mathbb{F},\mathbb{P}\right)$-martingale. We define
$\mathbf{\Gamma}(s)\cdot d\mathbf{\widetilde{N}}(s)\triangleq\sum_{j=1}^{L}\Gamma_{j}(s)d\widetilde{N}_{j}(s),$
for given D-dimensional process $\mathbf{\Gamma}=\left[\Gamma_1,\Gamma_2,...,\Gamma_L\right]$
and define
$\Theta(\alpha_{t})\triangleq\sum_{i=1}^{L}\Theta(i)\mathbb{I}_{\{\alpha_{t}=i\}}(t)$ for given D-dimensional vector $\mathbf{\Theta}=\left[\Theta(1),\Theta(2),...,\Theta(L)\right]$.
Let $\mathbb{B}$ be a given Banach space and denote $\mathcal{D}(\mathbb{B})\triangleq\left\{\mathbf{\Lambda}=\left[\Lambda(1),\Lambda(2),...,\Lambda(L)\right] \mid \Lambda(i) \in \mathbb{B}\text{, }  i\in \mathcal{S}\right\}.$
Let $\mathbb{R}^{n\times m}$ be the Euclidean space of all $ n \times m$  matrices and set $\mathbb{R}^{n}:=\mathbb{R}^{n\times 1}$ for simplicity. In addition, the set of all $ n\times n$ symmetric matrices is denoted by $\mathbb{S}^n$. Specially, the sets of all $ n\times n$ semi-positive definite matrices and positive definite matrices are denoted by $\overline{\mathbb{S}_{+}^n}$ and $\mathbb{S}_{+}^n$, respectively. For any $M, N \in \mathbb{S}^n$, we write $M \geqslant N$ (respectively, $M>N$) if $M-N$ is semi-positive definite (respectively, positive definite). Let $M^{\dag}$ denote the pseudo-inverse of a matrix $M\in\mathbb{R}^{n\times m}$, which  is equal to the inverse $M^{-1}$ of $M\in\mathbb{R}^{n\times n}$ if it exist.   In addition, the identity matrix of size $n$ is denoted by $I_{n}$. When no confusion arises, we often suppress the index $n$ and write $I$ instead of $I_{n}$.

Let $\mathcal{P}$ be the $\mathbb{F}$ predictable $\sigma$-field on $[0,\infty)\times\Omega$ and write $\varphi \in \mathbb{F}$ (respectively, $\varphi \in \mathcal{P}$) if it is $\mathbb{F}$-progressively measurable (respectively, $\mathcal{P}$-measurable). Then, for any Euclidean space $\mathbb{H}$, we introduce the following space:
$$
\begin{aligned}
&L_{\mathbb{G}}^{2,loc}(\mathbb{H}) =\left\{\varphi:[0, \infty)\times \Omega \rightarrow \mathbb{H} \mid \varphi \in \mathbb{G}\text{, } \mathbb{E} \int_{0}^{T}|\varphi(s)|^{2} ds<\infty, \quad\forall T>0\right\}, \\
&L_{\mathbb{G}}^{2}(\mathbb{H}) =\left\{\varphi:[0, \infty) \times \Omega \rightarrow \mathbb{H} \mid \varphi \in \mathbb{G}\text{, } \mathbb{E} \int_{0}^{\infty}|\varphi(s)|^{2} ds<\infty\right\},\quad
\mathbb{G}=\mathbb{F},\mathcal{P}.
\end{aligned}
$$

Before we formally introduce the LQ-SDG problem, we first provide some basic concepts for controlled linear stochastic differential equation (SDE, for short) as follows:
\begin{equation}\label{ACBD}
\left\{
\begin{aligned}
&dX(t)=\left[A(\alpha_{t})X(t)+B(\alpha_{t})u(t)\right]dt+\left[C(\alpha_{t})X(t)+D(\alpha_{t})u(t)\right]dW(t), \quad t\geq 0,\\
&X(0)=x,\quad \alpha_{0}=i,
\end{aligned}
\right.
\end{equation}
where $u\in L_{\mathbb{F}}^{2}(\mathbb{R}^{m})$ and for $i\in\mathcal{S}$, $A(i)$, $C(i)\in\mathbb{R}^{n\times n}$, $B(i)$, $D(i)\in\mathbb{R}^{n\times m}$. For simplicity, we denote the above system as system $[A,C;B,D]_{\alpha}$ and $[A,C]_{\alpha}$ for the case of $B=D=0$, i.e., $[A,C]_{\alpha}=[A,C;0,0]_{\alpha}$.
\begin{definition}\label{def-L2}
\begin{description}
\item[(i)] System $[A,C]_{\alpha}$ is said to be $L^{2}$-stable if for any $(x,i)\in\mathbb{R}^{n}\times\mathcal{S}$, the solution $X$ to \eqref{ACBD} with $B=D=0$ is in $L_{\mathbb{F}}^{2}(\mathbb{R}^{n})$.
\item[(ii)] System $[A,C;B,D]_{\alpha}$ is said to be $L^{2}$-stabilizable if and only if there exists a $\mathbf{\Theta}\in\mathcal{D}\left(\mathbb{R}^{m\times n}\right)$ such that the following system  is $L^2$-stable:
\begin{equation*}
\left\{
\begin{aligned}
&dX(t)=\left[A(\alpha_{t})+B(\alpha_{t})\Theta(\alpha_{t})\right]X(t)dt+\left[C(\alpha_{t})+D(\alpha_{t})\Theta(\alpha_{t})\right]X(t)dW(t), \quad t\geq 0,\\
&X(0)=x,\quad \alpha_{0}=i.
\end{aligned}
\right.
\end{equation*}
Any element $\mathbf{\Theta}$ satisfying the above condition is called a stabilizer of system $\left[A,C;B,D\right]_{\alpha}$ and we denote the set of all stabilizers as $\mathcal{H}_{\alpha}\equiv \mathcal{H}\left[A,C;B,D\right]_{\alpha}$.
\end{description}
\end{definition}


Now, we return to introduce the LQ-SDG problem, whose state process is a linear SDE and controlled by two players:
 \begin{equation}\label{GLQ-state}
   \left\{
   \begin{aligned}
   dX(t)&=\left[A\left(\alpha_{t}\right)X(t)+B_{1}\left(\alpha_{t}\right)u_{1}(t)+B_{2}\left(\alpha_{t}\right)u_{2}(t)+b(t)\right]dt\\
   &\quad+\left[C\left(\alpha_{t}\right)X(t)+D_{1}\left(\alpha_{t}\right)u_{1}(t)+D_{2}\left(\alpha_{t}\right)u_{2}(t)+\sigma(t)\right]dW(t),\qquad t\geq0,\\
   X(0)&=x,\quad \alpha(0)=i,
   \end{aligned}
   \right.
 \end{equation}
where $b$, $\sigma\in L_{\mathbb{F}}^{2}(\mathbb{R}^{n})$. For $i\in\mathcal{S}$, $A(i),\text{ } C(i)\in\mathbb{R}^{n\times n}$, $B_{k}(i),\text{ } D_{k}(i)\in\mathbb{R}^{n\times m_{k}}$, $k=1,\text{ }2$. In the above, $X\equiv X(\cdot;x,i,u_{1},u_{2})$, valued in $\mathbb{R}^{n}$, is called the \emph{state process} and $u_{k}$, valued in $\mathbb{R}^{m_{k}}$, is called the \emph{control process} of player $k$. The cost functional of player $k$ is defined as
\begin{equation}\label{GLQ-cost}
\begin{aligned}
    J_{k}\left(x,i;u_{1},u_{2}\right)
    \triangleq \mathbb{E}\int_{0}^{\infty}&\left[
    \left<
    \left(
    \begin{matrix}
    Q^{k}(\alpha_{t}) & S_{1}^{k}(\alpha_{t})^{\top} & S_{2}^{k}(\alpha_{t})^{\top} \\
    S_{1}^{k}(\alpha_{t}) & R_{11}^{k}(\alpha_{t}) & R_{12}^{k}(\alpha_{t}) \\
    S_{2}^{k}(\alpha_{t}) & R_{21}^{k}(\alpha_{t}) & R_{22}^{k}(\alpha_{t})
    \end{matrix}
    \right)
    \left(
    \begin{matrix}
    X(t) \\
    u_{1}(t) \\
    u_{2}(t)
    \end{matrix}
    \right),
    \left(
    \begin{matrix}
    X(t) \\
    u_{1}(t) \\
    u_{2}(t)
    \end{matrix}
    \right)
    \right>\right.\\
    &\left.+2\left<\left(\begin{matrix}
    q^{k}(t) \\
    \rho_{1}^{k}(t) \\
    \rho_{2}^{k}(t)
    \end{matrix}
    \right),
    \left(\begin{matrix}
    X(t) \\
    u_{1}(t) \\
    u_{2}(t)
    \end{matrix}
    \right)\right>\right]dt,\quad k=1,2,
  \end{aligned}
\end{equation}
where $q^{k}\in L_{\mathbb{F}}^{2}(\mathbb{R}^{n})$, $\rho_{1}^{k}\in L_{\mathbb{F}}^{2}(\mathbb{R}^{m_{1}})$, $\rho_{2}^{k}\in L_{\mathbb{F}}^{2}(\mathbb{R}^{m_{2}})$  and for $i\in\mathcal{S}$,
\begin{equation}\label{GLQ-cost-coefficients}
\begin{aligned}
 &Q^{k}(i) \in \mathbb{S}^{n}, \quad R_{ll}^{k}(i) \in \mathbb{S}^{m_{l}}, \quad S_{l}^{k}(i)\in\mathbb{R}^{m_{l} \times n},
 \quad R_{12}^{k}(i)=R_{21}^{k}(i)^{\top}\in\mathbb{R}^{m_{1} \times m_{2}},\quad k,l=1,2.
 \end{aligned}
\end{equation}
Here, we do not require that $Q^{k}(i)$, $R_{11}^{k}(i)$ and $R_{22}^{k}(i)$ in the \eqref{GLQ-cost-coefficients} are (semi) positive definite matrices. Hence, we are about to solve an indefinite LQ-SDG problem. For notations simplicity, let $m=m_{1}+m_{2}$ and denote (for $k=1,2$, $i\in\mathcal{S}$)
\begin{equation}\label{GLQ-open-notation-state+cost}
\begin{aligned}
 &B(i)\triangleq\left(B_{1}(i),B_{2}(i)\right),\quad  D(i)\triangleq\left(D_{1}(i),D_{2}(i)\right),\quad
 S^{k}(i)^{\top}\triangleq\left( S_{1}^{k}(i)^{\top}, S_{2}^{k}(i)^{\top}\right),\\
 &R^{k}(i)\triangleq\left(\begin{matrix} R_{1}^{k}(i)\\R_{2}^{k}(i)\end{matrix}\right)
  \triangleq\left(\begin{matrix} R_{11}^{k}(i) & R_{12}^{k}(i)\\R_{21}^{k}(i)& R_{22}^{k}(i)\end{matrix}\right)\quad
  \rho^{k}\triangleq\left(\begin{matrix} \rho_{1}^{k}\\\rho_{2}^{k}\end{matrix}\right)\quad
  u\triangleq\left(\begin{matrix} u_{1}\\u_{2}\end{matrix}\right).
\end{aligned}
\end{equation}
Then the state process \eqref{GLQ-state} and cost functional \eqref{GLQ-cost} can be rewritten as
\begin{equation}\label{state-open-simplicity}
   \left\{
   \begin{aligned}
   dX(t)&=\left[A\left(\alpha_{t}\right)X(t)+B\left(\alpha_{t}\right)u(t)+b(t)\right]dt+\left[C\left(\alpha_{t}\right)X(t)+D\left(\alpha_{t}\right)u(t)+\sigma(t)\right]dW(t),\\
   X(0)&=x,\quad \alpha(0)=i,\quad t\geq0,
   \end{aligned}
   \right.
 \end{equation}
and for $k=1,2,$
\begin{equation}\label{GLQ-cost-2}
\begin{aligned}
    J_{k}\left(x,i;u \right)
   = \mathbb{E}\int_{0}^{\infty}&\left[
    \left<
    \left(
    \begin{matrix}
    Q^{k}(\alpha_{t})&S^{k}(\alpha_{t})^{\top}\\
    S^{k}(\alpha_{t})&R^{k}(\alpha_{t})
    \end{matrix}
    \right)
    \left(
    \begin{matrix}
    X(s)\\
    u(s)
    \end{matrix}
    \right),
    \left(
    \begin{matrix}
    X(s)\\
    u(s)
    \end{matrix}
    \right)
    \right>+2\left<\left(\begin{matrix}
    q^{k}(s)\\
    \rho^{k}(s)
    \end{matrix}
    \right),
    \left(\begin{matrix}
    X(s)\\
    u(s)
    \end{matrix}
    \right)\right>\right]dt.
  \end{aligned}
\end{equation}

Generally, for any given initial value $(x,i)\in\mathbb{R}^{n}\times\mathcal{S}$ and control pair $(u_{1},u_{2})\in L_{\mathbb{F}}^{2}(\mathbb{R}^{m_{1}})\times L_{\mathbb{F}}^{2}(\mathbb{R}^{m_{2}})$, the  state process $X(\cdot;x,i,u_{1},u_{2})$ is usually in $L_{\mathbb{F}}^{2,loc}(\mathbb{R}^{n})$. To ensure the cost functional \eqref{GLQ-cost} is well-defined,  for any $(x,i)\in \mathbb{R}^{n}\times\mathcal{S}$, we need to introduce the following admissible control pair set:
$$\mathcal{U}_{ad}(x,i)\triangleq\left\{(u_{1},u_{2})\in L_{\mathbb{F}}^{2}(\mathbb{R}^{m_{1}})\times L_{\mathbb{F}}^{2}(\mathbb{R}^{m_{2}}) \mid X(\cdot;x,i,u_{1},u_{2})\in L_{\mathbb{F}}^{2}(\mathbb{R}^{n})\right\}. $$
A pair $(u_{1},u_{2})\in \mathcal{U}_{ad}(x,i)$ is called an admissible control pair for the initial state $(x,i)$. Now, we formulate the non-zero-sum LQ-SDG problem as follows.\\
\textbf{Problem (M-GLQ).} For any given $(x,i)\in \mathbb{R}^{n}\times \mathcal{S}$, find a $(u_{1}^{*},u_{2}^{*})\in\mathcal{U}_{ad}(x,i)$ such that
\begin{equation}\label{GLQ-value-function}
    \left\{
    \begin{aligned}
& J_{1}\left(x,i;u_{1}^{*},u_{2}^{*}\right)=\inf_{(u_{1},u_{2}^{*})\in \mathcal{U}_{ad}(x,i)}J_{1}\left(x,i;u_{1},u_{2}^{*}\right),\\
& J_{2}\left(x,i;u_{1}^{*},u_{2}^{*}\right)=\inf_{(u_{1}^{*},u_{2})\in \mathcal{U}_{ad}(x,i)}J_{2}\left(x,i;u_{1}^{*},u_{2}\right).
\end{aligned}
    \right.
\end{equation}
In addition, if  $b=\sigma=q^{k}=0$, $\rho_{1}^{k}=0$, $\rho_{2}^{k}=0$,  then the corresponding admissible control pair set, state process, cost functional and problem are denoted by $\mathcal{U}_{ad}^{0}(x,i)$, $ X^{0}(\cdot;x,i,u_{1},u_{2})$, $J_{k}^{0}(x,i;u_{1},u_{2})$ and Problem (M-GLQ)$^0$,  respectively.

In the LQ-SDG problem, for any initial value $(x,i)$, we require that the control pair $(u_{1},u_{2})$ composed by two players' control must make the corresponding state process $X(\cdot;x,i,u_{1},u_{2})$ is in
$L_{\mathbb{F}}^{2}(\mathbb{R}^{n})$. That is to say, for a given initial value
$(x,i)$, if player $l$ select a control
$u_{l}\in L_{\mathbb{F}}^{2}(\mathbb{R}^{m_{l}})$, then player
$k$ ($k\neq l$) must select their control in the following set:
$$\mathcal{U}_{ad}^{k}(x,i;u_{l})\triangleq\left\{u_{k}\in L_{\mathbb{F}}^{2}(\mathbb{R}^{m_{k}})|X(\cdot;x,i,u_{k},u_{l})\in L_{\mathbb{F}}^{2}(\mathbb{R}^{n})\right\},\quad (k,l)\in\{(1,2),(2,1)\}.$$
For a homogeneous LQ-SDG problem, one can similarly define the corresponding set $\mathcal{U}_{ad}^{k,0}(x,i;u_{l})$ as follows:
$$\mathcal{U}_{ad}^{k,0}(x,i;u_{l})\triangleq\left\{u_{k}\in L_{\mathbb{F}}^{2}(\mathbb{R}^{m_{k}})|X^{0}(\cdot;x,i,u_{k},u_{l})\in L_{\mathbb{F}}^{2}(\mathbb{R}^{n})\right\},\quad (k,l)\in\{(1,2),(2,1)\}.$$
Specifically, if there exist a $\mathbf{\Theta}_{l}\in\mathcal{D}(\mathbb{R}^{m_{l}\times n})$ and $\nu_{l}\in L_{\mathbb{F}}^{2}(\mathbb{R}^{m_{l}})$ such that $u_{l}=\Theta_{l}(\alpha)X+\nu_{l},$
then for any $u_{k}\in L_{\mathbb{F}}^{2}(\mathbb{R}^{m_{k}})$, we denote the corresponding state process as $X(\cdot;x,i,u_{k},\mathbf{\Theta}_{l},\nu_{l})$, that is to say, $X(\cdot;x,i,u_{k},\mathbf{\Theta}_{l},\nu_{l})$ solves the following SDE:
\begin{equation}\label{state-semi-closed}
    \left\{
    \begin{aligned}
    dX(t)&=\left[\big(A(\alpha_{t})+B_{l}(\alpha_{t})\Theta_{l}(\alpha_{t})\big)X(t)+B_{l}(\alpha_{t})\nu_{l}(t)+B_{k}(\alpha_{t})u_{k}(t)+b(t)\right]dt\\
   &\quad+\left[\big(C(\alpha_{t})+D_{l}(\alpha_{t})\Theta_{l}(\alpha_{t})\big)X(t)+D_{l}(\alpha_{t})\nu_{l}(t)+D_{k}(\alpha_{t})u_{k}(t)+\sigma(t)\right]dW(t),\\
   X(0)&=x,\quad \alpha(0)=i, \quad t\geq0,
    \end{aligned}
    \right.
\end{equation}
In this case, we simplify the set $\mathcal{U}_{ad}^{k}(x,i;u_{l})$ as follows:
$$\mathcal{U}_{ad}^{k}(x,i;\mathbf{\Theta}_{l},\nu_{l})\triangleq \left\{u_{k}\in L_{\mathbb{F}}^{2}(\mathbb{R}^{m_{k}})|X(\cdot;x,i,u_{k},\mathbf{\Theta}_{l},\nu_{l})\in L_{\mathbb{F}}^{2}(\mathbb{R}^{n})\right\},\quad (k,l)\in\{(1,2),(2,1)\}.$$
Furthermore, if there exist a $\mathbf{\Theta}=(\mathbf{\Theta}_{1}^{\top},\mathbf{\Theta}_{2}^{\top})^{\top}\in\mathcal{D}(\mathbb{R}^{m\times n})$ and $\nu=(\nu_{1}^{\top},\nu_{2}^{\top})^{\top}\in L_{\mathbb{F}}^{2}(\mathbb{R}^{m})$ such that
$$u_{k}=\Theta_{k}(\alpha)X+\nu_{k},\quad k=1,2,$$
then we denote the corresponding state process as $X(\cdot;x,i,\mathbf{\Theta},\nu)$, which is the solution of the following closed-form SDE:
\begin{equation}\label{state-closed}
    \left\{
    \begin{aligned}
    dX(t)&=\left[\big(A(\alpha_{t})+B(\alpha_{t})\Theta(\alpha_{t})\big)X(t)+B(\alpha_{t})\nu(t)+b(t)\right]dt\\
   &\quad+\left[\big(C(\alpha_{t})+D(\alpha_{t})\Theta(\alpha_{t})\big)X(t)+D(\alpha_{t})\nu(t)+\sigma(t)\right]dW(t), \quad t\geq0,\\
   X(0)&=x,\quad \alpha(0)=i,
    \end{aligned}
    \right.
\end{equation}

If $b=\sigma=q^{k}=0$, $\rho_{1}^{k}=0$, $\rho_{2}^{k}=0$, we denote the corresponding solutions to \eqref{state-semi-closed} and \eqref{state-closed} as $X^{0}(\cdot;x,i,u_{k},\mathbf{\Theta}_{l},\nu_{l})$ and  $X^{0}(\cdot;x,i,\mathbf{\Theta},\nu)$, respectively, and, for $k=1,2$, define
\begin{equation}\label{cost-functional-GLQ}
\left\{
\begin{aligned}
    &J_{k}(x,i;u_{k},\mathbf{\Theta}_{l},\nu_{l})\triangleq J_{k}(x,i;u_{k},\Theta_{l}(\alpha)X+\nu_{l})\quad \text{ with }\quad X\equiv X(\cdot;x,i,u_{k},\mathbf{\Theta}_{l},\nu_{l}),\\
    &J_{k}(x,i;\mathbf{\Theta},\nu)\triangleq J_{k}(x,i;\Theta_{1}(\alpha)X+\nu_{1},\Theta_{2}(\alpha)X+\nu_{2})\quad  \text{ with }\quad X\equiv X(\cdot;x,i,\mathbf{\Theta},\nu),\\
    &J_{k}^{0}(x,i;u_{k},\mathbf{\Theta}_{l},\nu_{l})\triangleq J_{k}^{0}(x,i;u_{k},\Theta_{l}(\alpha)X^{0}+\nu_{l})\quad \text{ with }\quad X^{0}\equiv X^{0}(\cdot;x,i,u_{k},\mathbf{\Theta}_{l},\nu_{l}),\\
    &J_{k}^{0}(x,i;\mathbf{\Theta},\nu)\triangleq J_{k}^{0}(x,i;\Theta_{1}(\alpha)X^{0}+\nu_{1},\Theta_{2}(\alpha)X^{0}+\nu_{2})\quad  \text{ with }\quad X^{0}\equiv X^{0}(\cdot;x,i,\mathbf{\Theta},\nu).
\end{aligned}\right.
\end{equation}

Under the above notations, we can define the open-loop solvability and closed-loop solvability as follows.
\begin{definition}\label{def-open-loop-equilibrium}
A pair $(u_{1}^{*},u_{2}^{*})\equiv(u_{1}^{*}(\cdot;x,i),u_{2}^{*}(\cdot;x,i))\in\mathcal{U}_{ad}(x,i)$ is called an open-loop Nash equilibrium point of Problem (M-GLQ) for the initial state $(x,i)\in\mathbb{R}^{n}\times\mathcal{S}$ if
\begin{equation}\label{open-loop-Nash-equilibrium-point}
    \left\{
    \begin{aligned}
    &J_{1}(x,i;u_{1}^{*},u_{2}^{*})\leq J_{1}(x,i;u_{1},u_{2}^{*}),\quad \forall u_{1}\in \mathcal{U}_{ad}^{1}(x,i;u_{2}^{*}),\\
    &J_{2}(x,i;u_{1}^{*},u_{2}^{*})\leq J_{2}(x,i;u_{1}^{*},u_{2}),\quad \forall u_{2}\in \mathcal{U}_{ad}^{2}(x,i;u_{1}^{*}).\\
    \end{aligned}
    \right.
\end{equation}
Problem (M-GLQ) is said to be (uniquely) open-loop solvable at $(x,i)$ if it admits a (unique)  open-loop Nash equilibrium point $(u_{1}^{*},u_{2}^{*})\in\mathcal{U}_{ad}(x,i)$. In addition, Problem (M-GLQ) is said to be (uniquely) open-loop solvable if it is  (uniquely) open-loop solvable at all $(x,i)\in\mathbb{R}^{n}\times\mathcal{S}$.
In this case, we call that the Problem (M-GLQ) admits a closed-loop  representation if there exists a pair $(\mathbf{\Theta}^{*},\nu^{*})\in\mathcal{H}[A,C;B,D]_{\alpha}\times L_{\mathbb{F}}^{2}(\mathbb{R}^{m})$ such that:
  $$u^{*}(\cdot;x,i)\triangleq (u_{1}^{*}(\cdot;x,i)^{\top},u_{2}^{*}(\cdot;x,i)^{\top})^{\top}=\Theta^{*}(\alpha)X(\cdot;x,i,\mathbf{\Theta}^{*},\nu^{*})+\nu^{*},\quad \forall (x,i)\in\mathbb{R}^{n}\times \mathcal{S}.$$
In the above, the pair $(\mathbf{\Theta}^{*},\nu^{*})$ is called a closed-loop representation strategy of Problem (M-GLQ).
\end{definition}

\begin{definition}\label{def-closed-loop-equilibrium}
A 4-tuple $(\mathbf{\widehat{\Theta}_{1}},\widehat{\nu}_{1};\mathbf{\widehat{\Theta}_{2}},\widehat{\nu}_{2})\in\mathcal{D}\left(\mathbb{R}^{m_{1}\times n}\right)\times L_{\mathbb{F}}^{2}(\mathbb{R}^{m_{1}})\times \mathcal{D}\left(\mathbb{R}^{m_{2}\times n}\right)\times L_{\mathbb{F}}^{2}(\mathbb{R}^{m_{2}})$ is called a closed-loop Nash equilibrium strategy of Problem (M-GLQ) if
\begin{description}
  \item[(i)] $\mathbf{\widehat{\Theta}}= (\mathbf{\widehat{\Theta}_{1}}^{\top},\mathbf{\widehat{\Theta}_{2}}^{\top})^{\top}\in \mathcal{H}\left[A,C;B,D\right]_{\alpha}$,
  \item[(ii)] for any $(x,i)\in\mathbb{R}^{n}\times\mathcal{S}$, the following holds:
  \begin{equation}\label{closed-loop-Nash-equilibrium-point}
  \left\{
      \begin{aligned}
&J_{1}(x,i;\mathbf{\widehat{\Theta}},\widehat{\nu})\leq J_{1}(x,i;u_{1},\mathbf{\widehat{\Theta}}_{2},\widehat{\nu}_{2}),\quad \forall u_{1}\in\mathcal{U}_{ad}^{1}(x,i;\mathbf{\widehat{\Theta}}_{2},\widehat{\nu}_{2}),\\
&J_{2}(x,i;\mathbf{\widehat{\Theta}},\widehat{\nu})\leq J_{2}(x,i;\mathbf{\widehat{\Theta}}_{1},\widehat{\nu}_{1},u_{2}),\quad \forall u_{2}\in\mathcal{U}_{ad}^{2}(x,i;\mathbf{\widehat{\Theta}}_{1},\widehat{\nu}_{1}),
      \end{aligned}
      \right.
  \end{equation}
where $\widehat{\nu}\triangleq (\nu_{1}^{\top},\nu_{2}^{\top})^{\top}\in L_{\mathbb{F}}^{2}(\mathbb{R}^{m})$.
\end{description}
In the above, we denote
\begin{equation}\label{GLQ-outcome-closed}
\left\{
\begin{aligned}
&\widehat{u}_{1}(\cdot;x,i,\mathbf{\widehat{\Theta}},\widehat{\nu})\triangleq \widehat{\Theta}_{1}(\alpha)X(\cdot;x,i,\mathbf{\widehat{\Theta}},\widehat{\nu})+\widehat{\nu}_{1},\\
&\widehat{u}_{2}(\cdot;x,i,\mathbf{\widehat{\Theta}},\widehat{\nu})\triangleq \widehat{\Theta}_{2}(\alpha)X(\cdot;x,i,\mathbf{\widehat{\Theta}},\widehat{\nu})+\widehat{\nu}_{2},\\
\end{aligned}
\right.
\end{equation}
as the outcome of the closed-loop Nash equilibrium strategy $(\mathbf{\widehat{\Theta}},\widehat{\nu})$.
In addition,
we denote the closed-loop Nash equilibrium  value function of player $k$ as
\begin{equation}\label{GLQ-closed-value}
V_{k}(x,i)\triangleq J_{k}(x,i;\mathbf{\widehat{\Theta}},\widehat{\nu}),\quad k=1,2,
\end{equation}
and the corresponding homogeneous closed-loop Nash equilibrium  value function of Player $k$ is denoted by $V_{k}^{0}(x,i)$.
\end{definition}

Plugging \eqref{GLQ-outcome-closed} into \eqref{closed-loop-Nash-equilibrium-point} and comparing with \eqref{open-loop-Nash-equilibrium-point}, one can easily find that the outcome of the closed-loop Nash equilibrium strategy $(\mathbf{\widehat{\Theta}},\widehat{\nu})$ is not an open-loop Nash equilibrium point of Problem (M-GLQ). Hence, unlike the LQ problem, the closed-loop solvability can not imply the open-loop solvability in the Problem (M-GLQ).

Suppose that the coefficients in cost functional \eqref{GLQ-cost} satisfies
\begin{equation}\label{ZLQ-cost-notation}
\left\{
  \begin{aligned}
&Q^{1}(i)=-Q^{2}(i)\triangleq Q(i),\quad q^{1}=-q^{2}\triangleq q,\\
&S^{1}(i)= -S^{2}(i)\triangleq S(i)=\left(\begin{matrix} S_{1}(i)\\ S_{2}(i)\end{matrix}\right),\quad
\rho^{1}= -\rho^{2}\triangleq \rho\equiv \left(\begin{matrix} \rho_{1}\\\rho_{2}\end{matrix}\right),\\
&R^{1}(i) = -R^{2}(i)\triangleq R(i)\equiv \left(\begin{matrix} R_{1}(i)\\R_{2}(i)\end{matrix}\right)
  = \left(\begin{matrix} R_{11}(i) & R_{12}(i)\\R_{21}(i) & R_{22}(i)\end{matrix}\right).
\end{aligned}
\right.
\end{equation}
Then we introduce the following performance functional:
\begin{equation}\label{zero-sum-performance}
\begin{aligned}
   & J\left(x,i;u_{1},u_{2}\right)\\
    &\triangleq \mathbb{E}\int_{0}^{\infty}\left[
    \left<
    \left(
    \begin{matrix}
    Q(\alpha_{t})&S_{1}(\alpha_{t})^{\top}&S_{2}(\alpha_{t})^{\top}\\
    S_{1}(\alpha_{t})&R_{11}(\alpha_{t})&R_{12}(\alpha_{t})\\
    S_{2}(\alpha_{t})&R_{21}(\alpha_{t})&R_{22}(\alpha_{t})
    \end{matrix}
    \right)
    \left(
    \begin{matrix}
    X(t)\\
    u_{1}(t)\\
    u_{2}(t)
    \end{matrix}
    \right),
    \left(
    \begin{matrix}
    X(t)\\
    u_{1}(t)\\
    u_{2}(t)
    \end{matrix}
    \right)
    \right>+2
    \left<
    \left(
    \begin{matrix}
    q(t)\\
    \rho_{1}(t)\\
    \rho_{2}(t)
    \end{matrix}
    \right),
    \left(
    \begin{matrix}
    X(t)\\
    u_{1}(t)\\
    u_{2}(t)
    \end{matrix}
    \right)
    \right>\right]dt\\
    &=\mathbb{E}\int_{0}^{\infty}\left[
    \left<
    \left(
    \begin{matrix}
    Q(\alpha_{t})&S(\alpha_{t})^{\top}\\
    S(\alpha_{t})&R(\alpha_{t})
    \end{matrix}
    \right)
    \left(
    \begin{matrix}
    X(t)\\
    u(t)
    \end{matrix}
    \right),
    \left(
    \begin{matrix}
    X(t)\\
    u(t)
    \end{matrix}
    \right)
    \right>
+2\left<
    \left(
    \begin{matrix}
    q(t)\\
    \rho(t)
    \end{matrix}
    \right),
    \left(
    \begin{matrix}
    X(t)\\
    u(t)
    \end{matrix}
    \right)
    \right>\right]dt.
  \end{aligned}
\end{equation}
Obviously,
$$J_{1}\left(x,i;u_{1},u_{2}\right)=J\left(x,i;u_{1},u_{2}\right)=-J_{2}\left(x,i;u_{1},u_{2}\right).$$
We denote a non-zero-sum LQ-SDG problem with constraint \eqref{ZLQ-cost-notation} as a zero-sum LQ-SDG problem, whose mathematical definition is provided as follows.\\
\textbf{Problem (M-ZLQ).} For any given $(x,i)\in \mathbb{R}^{n}\times \mathcal{S}$, find a $(u_{1}^{*},u_{2}^{*})\in\mathcal{U}_{ad}(x,i)$ such that
\begin{equation}\label{ZLQ-value-function}
\sup_{(u_{1}^{*},u_{2})\in \mathcal{U}_{ad}(x,i)}J\left(x,i;u_{1}^{*},u_{2}\right)=J\left(x,i;u_{1}^{*},u_{2}^{*}\right)=\inf_{(u_{1},u_{2}^{*})\in \mathcal{U}_{ad}(x,i)}J\left(x,i;u_{1},u_{2}^{*}\right).\\
\end{equation}
In addition, if  $b=\sigma=q=0$, $\rho_{1}=0$, $\rho_{2}=0$,  then the corresponding performance functional and problem are denoted by  $J^{0}(x,i;u_{1},u_{2})$ and Problem (M-ZLQ)$^0$,  respectively.

Under the framework of zero-sum Nash equilibrium point game, player 1 wishes to minimize \eqref{zero-sum-performance} by selecting a control $u_{1}$, while player 2 wishes to maximize \eqref{zero-sum-performance} by selecting a control $u_{2}$. Hence, the performance functional \eqref{zero-sum-performance} represents the cost for Player 1 and the payoff for Player 2. One can define the $J(x,i;u_{k},\mathbf{\Theta}_{l},\nu_{l})$, $J(x,i;\mathbf{\Theta},\nu)$, $J^{0}(x,i;u_{k},\mathbf{\Theta}_{l},\nu_{l})$ and $J^{0}(x,i;\mathbf{\Theta},\nu)$ similar to \eqref{cost-functional-GLQ}. The open-loop solvability and closed-loop solvability of Problem (M-ZLQ) are defined as follows.
\begin{definition}\label{def-ZLQ-open-loop-equilibrium}
A pair $(u_{1}^{*},u_{2}^{*})\equiv (u_{1}^{*}(\cdot;x,i),u_{2}^{*}(\cdot;x,i))\in\mathcal{U}_{ad}(x,i)$ is called an open-loop saddle point of Problem (M-ZLQ) for the initial state $(x,i)\in\mathbb{R}^{n}\times\mathcal{S}$ if
\begin{equation}\label{ZLQ-open-loop-Nash-equilibrium-point}
    \begin{aligned}
    J\left(x,i;u_{1}^{*},u_{2}\right)\leq J\left(x,i;u_{1}^{*},u_{2}^{*}\right)\leq J\left(x,i;u_{1},u_{2}^{*}\right),\quad \forall \left(u_{1},u_{2}\right)\in \mathcal{U}_{ad}^{1}(x,i;u_{2}^{*})\times \mathcal{U}_{ad}^{2}(x,i;u_{1}^{*}).
    \end{aligned}
\end{equation}
Problem (M-ZLQ) is said to be (uniquely) open-loop solvable at $(x,i)$ if it admits a (unique)  open-loop saddle point $(u_{1}^{*},u_{2}^{*})\in\mathcal{U}_{ad}(x,i)$. In addition, Problem (M-ZLQ) is said to be (uniquely) open-loop solvable if it is  (uniquely) open-loop solvable at all $(x,i)\in\mathbb{R}^{n}\times\mathcal{S}$.
In this case, we call that the Problem (M-ZLQ) admits a closed-loop  representation if there exists a pair $(\mathbf{\Theta}^{*},\nu^{*})\in\mathcal{H}[A,C;B,D]_{\alpha}\times L_{\mathbb{F}}^{2}(\mathbb{R}^{m})$ such that:
  $$u^{*}(\cdot;x,i)=(u_{1}^{*}(\cdot;x,i)^{\top},u_{2}^{*}(\cdot;x,i)^{\top})^{\top}= \Theta^{*}(\alpha)X(\cdot;x,i,\mathbf{\Theta}^{*},\nu^{*})+\nu^{*},\quad \forall (x,i)\in\mathbb{R}^{n}\times \mathcal{S}.$$
   In the above, the pair $(\mathbf{\Theta}^{*},\nu^{*})$ is called a closed-loop representation strategy of Problem (M-ZLQ).
\end{definition}

\begin{definition}\label{def-ZLQ-closed-loop-equilibrium}
A 4-tuple $(\mathbf{\widehat{\Theta}_{1}},\widehat{\nu}_{1};\mathbf{\widehat{\Theta}_{2}},\widehat{\nu}_{2})\in\mathcal{D}\left(\mathbb{R}^{m_{1}\times n}\right)\times L_{\mathbb{F}}^{2}(\mathbb{R}^{m_{1}})\times \mathcal{D}\left(\mathbb{R}^{m_{2}\times n}\right)\times L_{\mathbb{F}}^{2}(\mathbb{R}^{m_{2}})$ is called a closed-loop Nash equilibrium strategy of Problem (M-ZLQ) if
\begin{description}
  \item[(i)] $\mathbf{\widehat{\Theta}}= (\mathbf{\widehat{\Theta}_{1}}^{\top},\mathbf{\widehat{\Theta}_{2}}^{\top})^{\top}\in \mathcal{H}\left[A,C;B,D\right]_{\alpha}$,
  \item[(ii)] for any $(x,i)\in\mathbb{R}^{n}\times\mathcal{S}$, the following holds:
  \begin{equation}\label{ZLQ-closed-loop-Nash-equilibrium-point}
      \begin{aligned}
&J(x,i;\mathbf{\widehat{\Theta}}_{1},\widehat{\nu}_{1},u_{2})\leq J(x,i;\mathbf{\widehat{\Theta}},\widehat{\nu})\leq J(x,i;u_{1},\mathbf{\widehat{\Theta}}_{2},\widehat{\nu}_{2}),\\
&\quad \forall (u_{1},u_{2})\in\mathcal{U}_{ad}^{1}(x,i;\mathbf{\widehat{\Theta}}_{2},\widehat{\nu}_{2})\times \mathcal{U}_{ad}^{2}(x,i;\mathbf{\widehat{\Theta}}_{1},\widehat{\nu}_{1}),
      \end{aligned}
  \end{equation}
where $\widehat{\nu}\equiv (\nu_{1}^{\top},\nu_{2}^{\top})^{\top}\in L_{\mathbb{F}}^{2}(\mathbb{R}^{m})$.
\end{description}
In this case, we denote
$$V(x,i)\triangleq J\left(x,i;\mathbf{\widehat{\Theta}},\widehat{\nu}\right),$$
as the closed-loop Nash equilibrium value function of Problem (M-ZLQ) and the corresponding homogeneous closed-loop Nash equilibrium value function is denoted by $V^{0}(x,i)$.
\end{definition}

\section{Non-zero-sum  Nash differential game}\label{section-GLQ}
We now begin to study the Problem (M-GLQ). Both open-loop solvability and closed-loop solvability are obtained in this section. We also construct the closed-loop representation of open-loop Nash equilibrium point when the Problem (M-GLQ) is open-loop solvable.
\subsection{Open-loop solvability for Problem (M-GLQ)}\label{subsection-GLQ-open}
We first discuss the open-loop solvability for Problem (M-GLQ).
The following result characterizes the open-loop Nash equilibrium point in terms of FBSDEs, whose proof depends mainly on Theorem $4.1$ in Wu et al. \cite{Wu-etal}.

\begin{theorem}\label{thm-GLQ-open-solvability}
 Suppose that systems $[A,C;B_{1},D_{1}]_{\alpha}$ and $[A,C;B_{2},D_{2}]_{\alpha}$ are $L^{2}$-stabilizable. Then $(u_{1}^{*},u_{2}^{*})\in\mathcal{U}_{ad}(x,i)$ is an open-loop Nash equilibrium point of Problem (M-GLQ) for initial value $(x,i)\in\mathbb{R}^{n}\times\mathcal{S}$ if and only if
 \begin{description}
   \item[(i)] The following convexity conditions hold:
   \begin{equation}\label{GLQ-convexity}
       \left\{
       \begin{array}{l}
         J_{1}^{0}(0,i;u_{1},0)\geq 0,\quad \forall u_{1}\in\mathcal{U}_{ad}^{1,0}(x,i;0), \quad \forall i\in\mathcal{S}, \\
         J_{2}^{0}(0,i;0,u_{2})\geq 0,\quad \forall u_{2}\in\mathcal{U}_{ad}^{2,0}(x,i;0),  \quad \forall i\in\mathcal{S},
       \end{array}
       \right.
   \end{equation}
   \item[(ii)] The adapted solution $\left(X^{*},Y_{k}^{*},Z_{k}^{*},\mathbf{\Gamma}_{k}^{*}\right)\in L_{\mathbb{F}}^{2}(\mathbb{R}^{n})\times L_{\mathbb{F}}^{2}(\mathbb{R}^{n})\times L_{\mathbb{F}}^{2}(\mathbb{R}^{n})\times\mathcal{D}\left(L_{\mathcal{P}}^{2}(\mathbb{R}^{n})\right)$ to the  FBSDEs
  \begin{equation}\label{FBSDEs-GLQ}
  \left\{
      \begin{aligned}
      dX^{*}(t)&=\left[A(\alpha_{t})X^{*}(t)+B(\alpha_{t})u^{*}(t)+b(t)\right]dt+\left[C(\alpha_{t})X^{*}(t)+D(\alpha_{t})u^{*}(t)+\sigma(t)\right]dW(t),\\
      dY_{k}^{*}(t)&=-\left[A(\alpha_{t})^{\top}Y_{k}^{*}(t)+C(\alpha_{t})^{\top}Z_{k}^{*}(t)+Q^{k}(\alpha_{t})X^{*}(t)+S^{k}(\alpha_{t})^{\top}u^{*}(t)+q^{k}(t)\right]dt\\
      &\quad+Z_{k}^{*}(t)dW(t)+\mathbf{\Gamma}_{k}^{*}(t)\cdot d\mathbf{\widetilde{N}}(t),\quad t\geq 0,\quad k=1,2,\\
      X^{*}(0)&=x,\quad\alpha_{0}=i,
      \end{aligned}
      \right.
  \end{equation}
  satisfies the  stationary condition:
  \begin{equation}\label{stationary-GLQ}
     B_{k}(\alpha_{t})^{\top}Y_{k}^{*}(t)+ D_{k}(\alpha_{t})^{\top}Z_{k}^{*}(t)+S_{k}^{k}(\alpha_{t})X^{*}(t)+R_{k}^{k}(\alpha_{t})u^{*}(t)+\rho_{k}^{k}(t)=0,\, k=1,2.
  \end{equation}
 \end{description}
\end{theorem}
\begin{proof}
For any given $(u_{1}^{*},u_{2}^{*})\in\mathcal{U}_{ad}(x,i)$, let $X_{1}\equiv X(\cdot;x,i,u_{1},u_{2}^{*})$, $X_{2}\equiv X(\cdot;x,i,u_{1}^{*},u_{2})$,
 \begin{align*}
&\bar{J}_{1}(x,i;u_{1})\!\triangleq\! \mathbb{E}\int_{0}^{\infty}\!\left[\!\left<\!\left(\!\begin{matrix}
Q^{1}(\alpha) & S_{1}^{1}(\alpha)^{\top}\\
S_{1}^{1}(\alpha) & R_{11}^{1}(\alpha)\end{matrix}\!\right)\!\left(\!\begin{matrix}
X_{1}\\u_{1}\end{matrix}\!\right),
\left(\!\begin{matrix}
X_{1}\\u_{1}\end{matrix}\!\right)\!\right>
\!+\!2\left<\!\left(\!\begin{matrix}
q^{1}+S_{2}^{1}(\alpha)^{\top}u_{2}^{*}\\
\rho_{1}^{1}+R_{12}^{1}(\alpha)u_{2}^{*}\end{matrix}\!\right),
\left(\!\begin{matrix}
X_{1}\\u_{1}\end{matrix}\!\right)\!\right>\!\right]dt,
  \end{align*}
and
\begin{align*}
&\bar{J}_{2}(x,i;u_{2})\!\triangleq\! \mathbb{E}\int_{0}^{\infty}\left[\!\left<\!\left(\!\begin{matrix}
Q^{2}(\alpha) & S_{2}^{2}(\alpha)^{\top}\\
S_{2}^{2}(\alpha) & R_{22}^{2}(\alpha)\end{matrix}\!\right)\left(\!\begin{matrix}
X_{2}\\u_{2}\end{matrix}\!\right),
\left(\!\begin{matrix}
X_{2}\\u_{2}\end{matrix}\!\right)\!\right>
\!+\!2\left<\!\left(\!\begin{matrix}
q^{2}+S_{1}^{2}(\alpha)^{\top}u_{1}^{*}\\
\rho_{2}^{2}+R_{21}^{2}(\alpha)u_{1}^{*}\end{matrix}\!\right),
\left(\!\begin{matrix}
X_{2}\\u_{2}\end{matrix}\!\right)\!\right>\!\right]dt.
  \end{align*}
Then,
  \begin{align*}
& J_{1}(x,i;u_{1},u_{2}^{*}) =\mathbb{E}\int_{0}^{\infty}\big<R_{22}^{1}(\alpha)u_{2}^{*}+2\rho_{2}^{1},u_{2}^{*}\big>dt +\bar{J}_{1}(x,i;u_{1}),\\
& J_{2}(x,i;u_{1}^{*},u_{2}) =\mathbb{E}\int_{0}^{\infty}\big<R_{11}^{2}(\alpha)u_{1}^{*}+2\rho_{1}^{2},u_{1}^{*}\big>dt +\bar{J}_{2}(x,i;u_{2}).
  \end{align*}
  Obviously, a pair $(u_{1}^{*},u_{2}^{*})$ is an open-loop Nash equilibrium point of Problem (M-GLQ) for initial value $(x,i)$ if and only if, for $k=1,2$, $u_{k}^{*}$ is an open-loop optimal control for an LQ control problem with state constraint $X_{k}$ and cost functional $\bar{J}_{k}(x,i;u_{k})$.
  Note that
 \begin{align*}
\bar{J}_{k}^{0}(0,i;u_{k})&\triangleq \mathbb{E}\int_{0}^{\infty}\left<\left(\begin{matrix}
Q^{k}(\alpha) & S_{k}^{k}(\alpha)^{\top}\\
S_{k}^{k}(\alpha) & R_{kk}^{k}(\alpha)\end{matrix}\right)\left(\begin{matrix}
X_{k}^{0}\\u_{k}\end{matrix}\right),
\left(\begin{matrix}
X_{k}^{0}\\u_{k}\end{matrix}\right)\right>dt
=J^{0}(0,i;u_{k},0),\quad k=1,2,
  \end{align*}
where $X_{k}^{0}=X^{0}(\cdot;0,i,u_{k},0)$.
Hence,  the desired result follows from the Theorem  $4.1$ in Wu et al. \cite{Wu-etal}.
\end{proof}
\begin{remark}\label{rmk-GLQ-open-FBSDEs-1}
Obviously, a sufficient condition to ensure that the convexity condition \eqref{GLQ-convexity} holds is
\begin{equation}\label{GLQ-convexity-condition}
\left(\begin{matrix}
Q^{1}(i) & S_{1}^{1}(i)^{\top}\\
S_{1}^{1}(i) & R_{11}^{1}(i)\end{matrix}\right)\geq 0,\quad
\left(\begin{matrix}
Q^{2}(i) & S_{2}^{2}(i)^{\top}\\
S_{2}^{2}(i) & R_{22}^{2}(i)\end{matrix}\right)\geq 0,\quad \forall i\in\mathcal{S},
\end{equation}
Li et al. \cite{Li-Shi-Yong-2021-ID-MFLQ-IF} has investigated an LQ-SDG problem with a mean field in infinite horizon. Using the variational method, they also obtain a similar result. However, the convexity condition summarized in \cite{Li-Shi-Yong-2021-ID-MFLQ-IF} is too abstract and can not be verified directly. Hence, the condition \eqref{GLQ-convexity-condition} can be considered as an improvement compared with \cite{Li-Shi-Yong-2021-ID-MFLQ-IF}.
\end{remark}

Note that \eqref{FBSDEs-GLQ} together with \eqref{stationary-GLQ} constitute a coupled FBSDEs, which is regarded as the optimality system of Problem (M-GLQ). Our next task is to decouple the system \eqref{FBSDEs-GLQ}-\eqref{stationary-GLQ} and find the closed-loop representation of open-loop Nash equilibrium point for Problem (M-GLQ). To this end, we introduce the following notations:
\begin{equation}\label{GLQ-open-notation-2}
 \begin{aligned}
 &\mathbb{A}(i)=\left(\begin{matrix}A(i) & 0 \\ 0 & A(i) \end{matrix}\right), \quad
 \mathbb{C}(i)=\left(\begin{matrix}C(i) & 0 \\ 0 & C(i) \end{matrix}\right)\in \mathbb{R}^{2n\times 2n}, \quad
 \mathbb{I}_{m}=\left(\begin{matrix}I_{m}\\I_{m} \end{matrix}\right)\in\mathbb{R}^{2m\times m},\\
 &\mathbb{B}(i)=\left(\begin{matrix}B(i) & 0 \\ 0 & B(i) \end{matrix}\right), \quad
 \mathbb{D}(i)=\left(\begin{matrix}D(i) & 0 \\ 0 & D(i) \end{matrix}\right)\in \mathbb{R}^{2n\times 2m}, \quad \mathbb{I}_{n}=\left(\begin{matrix}I_{n}\\I_{n} \end{matrix}\right)\in\mathbb{R}^{2n\times n},\\
 &\mathbb{Q}(i)=\left(\begin{matrix}Q^{1}(i) & 0 \\ 0 & Q^{2}(i) \end{matrix}\right)\in \mathbb{S}^{2n}, \quad
 \mathbb{S}(i)=\left(\begin{matrix}S^{1}(i) & 0 \\ 0 & S^{2}(i) \end{matrix}\right)\in \mathbb{R}^{2m\times 2n}, \\
 &\mathbb{R}(i)=\left(\begin{matrix}R^{1}(i) & 0 \\ 0 & R^{2}(i) \end{matrix}\right)\in \mathbb{S}^{2m}, \quad
 q=\left(\begin{matrix}q^{1} \\ q^{2} \end{matrix}\right)\in L_{\mathbb{F}}^{2}(\mathbb{R}^{2n}), \quad
 \rho=\left(\begin{matrix}\rho^{1} \\ \rho^{2} \end{matrix}\right)\in L_{\mathbb{F}}^{2}(\mathbb{R}^{2m}).
 \end{aligned}
\end{equation}
Then the FBSDEs \eqref{FBSDEs-GLQ} and corresponding stationary condition \eqref{stationary-GLQ} can be rewritten as:
\begin{equation}\label{FBSDEs-GLQ-2}
  \left\{
      \begin{aligned}
      d\mathbb{X}^{*}(t)&=\left[\mathbb{A}(\alpha)\mathbb{X}^{*}+\mathbb{B}(\alpha)\mathbb{I}_{m}u^{*}+\mathbb{I}_{n}b\right]dt
      +\left[\mathbb{C}(\alpha)\mathbb{X}^{*}+\mathbb{D}(\alpha)\mathbb{I}_{m}u^{*}+\mathbb{I}_{n}\sigma\right]dW(t),\\
      d\mathbb{Y}^{*}(t)&=-\left[\mathbb{A}(\alpha)^{\top}\mathbb{Y}^{*}+\mathbb{C}(\alpha)^{\top}\mathbb{Z}^{*}+\mathbb{Q}(\alpha)\mathbb{X}^{*}+\mathbb{S}(\alpha)^{\top}\mathbb{I}_{m}u^{*}+q\right]dt+\mathbb{Z}^{*}dW(t)+\mathbf{\Gamma}^{*}\cdot d\mathbf{\widetilde{N}}(t),\\
      \mathbb{X}^{*}(0)&=\mathbb{I}_{n}x,\quad\alpha_{0}=i,\quad t\geq 0,
      \end{aligned}
      \right.
  \end{equation}
  and
\begin{equation}\label{stationary-GLQ-2}
     \mathbb{J}^{\top}\left\{\mathbb{B}(\alpha_{t})^{\top}\mathbb{Y}^{*}(t)+ \mathbb{D}(\alpha_{t})^{\top}\mathbb{Z}^{*}(t)+\mathbb{S}(\alpha_{t})\mathbb{X}^{*}(t)+\mathbb{R}(\alpha_{t})\mathbb{I}_{m}u^{*}(t)+\rho(t)\right\}=0,\quad a.e.\quad a.s.,
  \end{equation}
where
  \begin{align*}
  &\mathbb{X}^{*}=\mathbb{I}_{n}X^{*},\quad \mathbb{Y}^{*}=\left(\begin{matrix}Y_{1}^{*}\\Y_{2}^{*} \end{matrix}\right)\in L_{\mathbb{F}}^{2}(\mathbb{R}^{2n}),\quad
  \mathbb{Z}^{*}=\left(\begin{matrix}Z_{1}^{*}\\Z_{2}^{*} \end{matrix}\right)\in L_{\mathbb{F}}^{2}(\mathbb{R}^{2n}),\\
  &\mathbf{\Gamma}^{*}=\left(\begin{matrix}\mathbf{\Gamma}_{1}^{*}\\\mathbf{\Gamma}_{2}^{*} \end{matrix}\right)\in\mathcal{D}\left( L_{\mathcal{P}}^{2}(\mathbb{R}^{2n})\right),\quad
  \mathbb{J}^{\top}=\left(\begin{matrix}
  I_{m_{1}} & 0_{m_{1}\times m_{2}} & 0_{m_{1}\times m_{1}} & 0_{m_{1}\times m_{2}}\\
  0_{m_{2}\times m_{1}} & 0_{m_{2}\times m_{2}} & 0_{m_{2}\times m_{1}} & I_{m_{2}} \end{matrix}\right)\in\mathbb{R}^{m\times 2m}.
  \end{align*}

Now, we use the well-known four-step method to decouple the FBSDEs \eqref{FBSDEs-GLQ-2}-\eqref{stationary-GLQ-2}.
We first take the following ansatz:
\begin{equation}\label{decouple-ansatz}
    \mathbb{Y}^{*}=\mathbb{P}(\alpha)\mathbb{X}^{*}+\eta,
\end{equation}
where, for any $i\in\mathcal{S}$,
\begin{equation}
    \mathbb{P}(i)=\left(\begin{matrix}P_{1}(i) & 0\\ 0 & P_{2}(i)\end{matrix}\right)\in \mathbb{R}^{2n\times 2n},\quad
    \eta=\left(\begin{matrix}\eta_{1}\\\eta_{1} \end{matrix}\right)\in L_{\mathbb{F}}^{2}(\mathbb{R}^{2n}).
\end{equation}
Additionally,  $(\eta,\zeta,\mathbf{z})\in L_{\mathbb{F}}^{2}(\mathbb{R}^{2n})\times L_{\mathbb{F}}^{2}(\mathbb{R}^{2n})\times \mathcal{D}\left(L_{\mathcal{P}}^{2}(\mathbb{R}^{2n})\right) $ solves the BSDE:
\begin{equation}
    d\eta(t)=\beta(t)dt+\zeta(t)dW(t)+\mathbf{z}(t)\cdot d\widetilde{\mathbf{N}}(t),\quad t\geq 0,
\end{equation}
where $\beta$ is an undetermined process. Applying It\^o's rule to \eqref{decouple-ansatz} and comparing with \eqref{FBSDEs-GLQ-2}, one has
\begin{equation}\label{decouple-1}
\begin{aligned}
0&=\mathbb{A}(\alpha_{t})^{\top}\mathbb{Y}^{*}(t)+\mathbb{C}(\alpha_{t})^{\top}\mathbb{Z}^{*}(t)+\mathbb{Q}(\alpha_{t})\mathbb{X}^{*}(t)+\mathbb{S}(\alpha_{t})^{\top}\mathbb{I}_{m}u^{*}(t)+q(t)\\
&\quad +\mathbb{P}(\alpha_{t})\left[\mathbb{A}(\alpha_{t})\mathbb{X}^{*}(t)+\mathbb{B}(\alpha_{t})\mathbb{I}_{m}u^{*}(t)+\mathbb{I}_{n}b(t)\right]+\sum_{j=1}^{L}\pi_{\alpha_{t}j}\mathbb{P}(j)\mathbb{X}^{*}(t)+\beta(t),
\end{aligned}
\end{equation}
\begin{equation}\label{decouple-2}
\begin{aligned}
\mathbb{Z}^{*}(t)&=\mathbb{P}(\alpha_{t})\left[\mathbb{C}(\alpha_{t})\mathbb{X}^{*}(t)+\mathbb{D}(\alpha_{t})\mathbb{I}_{m}u^{*}(t)+\mathbb{I}_{n}\sigma(t)\right]+\zeta(t),
\end{aligned}
\end{equation}
and
\begin{equation}\label{decouple-3}
\begin{aligned}
\Gamma_{j}^{*}(t)&=\left[\mathbb{P}(j)-\mathbb{P}(\alpha_{t-})\right]\mathbb{X}^{*}(t-)+z_{j}(t),\quad j\in\mathcal{S}.
\end{aligned}
\end{equation}
Plugging \eqref{decouple-ansatz} and \eqref{decouple-2} into \eqref{stationary-GLQ-2}, we have
\begin{equation}\label{decouple-4}
  \begin{aligned}
0&=\mathbb{J}^{\top}\left\{\mathbb{B}(\alpha_{t})^{\top}\left[\mathbb{P}(\alpha_{t})\mathbb{X}^{*}(t)+\eta(t)\right]+\mathbb{S}(\alpha_{t})\mathbb{X}^{*}(t)+\mathbb{R}(\alpha_{t})\mathbb{I}_{m}u^{*}(t)+\rho(t)\right.\\
&\quad\left.+ \mathbb{D}(\alpha_{t})^{\top}\left[\mathbb{P}(\alpha_{t})\left(\mathbb{C}(\alpha_{t})\mathbb{X}^{*}(t)+\mathbb{D}(\alpha_{t})\mathbb{I}_{m}u^{*}(t)+\mathbb{I}_{n}\sigma(t)\right)+\zeta(t)\right] \right\}  \\
&=\mathbb{J}^{\top}\left\{ \big[\mathbb{B}(\alpha_{t})^{\top}\mathbb{P}(\alpha_{t})+ \mathbb{D}(\alpha_{t})^{\top}\mathbb{P}(\alpha_{t})\mathbb{C}(\alpha_{t})+\mathbb{S}(\alpha_{t})\big]\mathbb{X}^{*}(t)+\mathbb{B}(\alpha_{t})^{\top}\eta(t)\right.\\
&\quad\left. + \mathbb{D}(\alpha_{t})^{\top}\zeta(t)+ \mathbb{D}(\alpha_{t})^{\top}\mathbb{P}(\alpha_{t})\mathbb{I}_{n}\sigma(t)+\rho(t)\right\}+\Sigma(\mathbb{P},\alpha_{t})u^{*}(t),
\end{aligned}
\end{equation}
where
\begin{equation}\label{decouple-Sigma}
\Sigma(\mathbb{P},i)=\mathbb{J}^{\top}\left[\mathbb{R}(i)+ \mathbb{D}(i)^{\top}\mathbb{P}(i)\mathbb{D}(i)\right]\mathbb{I}_{m}.
\end{equation}
In the following, we assume that $\Sigma(\mathbb{P},i)$ is invertible for any $i\in\mathcal{S}$. Then, it follows from \eqref{decouple-4} that
\begin{equation}\label{decouple-u}
\begin{aligned}
u^{*}(t)&=-\Sigma(\mathbb{P},\alpha_{t})^{-1}\mathbb{J}^{\top}\left\{\big[\mathbb{B}(\alpha_{t})^{\top}\mathbb{P}(\alpha_{t})+ \mathbb{D}(\alpha_{t})^{\top}\mathbb{P}(\alpha_{t})\mathbb{C}(\alpha_{t})+\mathbb{S}(\alpha_{t})\big]\mathbb{X}^{*}(t)+\mathbb{B}(\alpha_{t})^{\top}\eta(t)\right.\\
&\quad\left. + \mathbb{D}(\alpha_{t})^{\top}\zeta(t)+ \mathbb{D}(\alpha_{t})^{\top}\mathbb{P}(\alpha_{t})\mathbb{I}_{n}\sigma(t)+\rho(t)\right\}.
\end{aligned}
\end{equation}
Substituting \eqref{decouple-u}, \eqref{decouple-2} and \eqref{decouple-ansatz} into \eqref{decouple-1}, one can easily obtain
\begin{align*}
0
&=\big[\mathbb{A}(\alpha_{t})^{\top}\mathbb{P}(\alpha_{t})+\mathbb{P}(\alpha_{t})\mathbb{A}(\alpha_{t})+\mathbb{C}(\alpha_{t})^{\top}\mathbb{P}(\alpha_{t})\mathbb{C}(\alpha_{t})+\mathbb{Q}(\alpha_{t})+\sum_{j=1}^{L}\pi_{\alpha_{t}j}\mathbb{P}(j)\big]\mathbb{X}^{*}(t)\\
&\quad +\big[\mathbb{P}(\alpha_{t})\mathbb{B}(\alpha_{t})+\mathbb{C}(\alpha_{t})^{\top}\mathbb{P}(\alpha_{t})\mathbb{D}(\alpha_{t})+\mathbb{S}(\alpha_{t})^{\top}\big]\mathbb{I}_{m}u^{*}(t)+\mathbb{A}(\alpha_{t})^{\top}\eta(t)+\mathbb{C}(\alpha_{t})^{\top}\zeta(t)\\
&\quad +\mathbb{C}(\alpha_{t})^{\top}\mathbb{P}(\alpha_{t})\mathbb{I}_{n}\sigma(t)+q(t)+\mathbb{P}(\alpha_{t})\mathbb{I}_{n}b(t)+\beta(t)\\
&=\Big\{\mathbb{A}(\alpha_{t})^{\top}\mathbb{P}(\alpha_{t})+\mathbb{P}(\alpha_{t})\mathbb{A}(\alpha_{t})+\mathbb{C}(\alpha_{t})^{\top}\mathbb{P}(\alpha_{t})\mathbb{C}(\alpha_{t})+\mathbb{Q}(\alpha_{t})+\sum_{j=1}^{L}\pi_{\alpha_{t}j}\mathbb{P}(j)\\
&\quad-\big[\mathbb{P}(\alpha_{t})\mathbb{B}(\alpha_{t})+\mathbb{C}(\alpha_{t})^{\top}\mathbb{P}(\alpha_{t})\mathbb{D}(\alpha_{t})+\mathbb{S}(\alpha_{t})^{\top}\big]\mathbb{I}_{m}\Sigma(\mathbb{P},\alpha_{t})^{-1}\mathbb{J}^{\top}\\
&\quad\times\big[\mathbb{B}(\alpha_{t})^{\top}\mathbb{P}(\alpha_{t})+ \mathbb{D}(\alpha_{t})^{\top}\mathbb{P}(\alpha_{t})\mathbb{C}(\alpha_{t})+\mathbb{S}(\alpha_{t})\big]\Big\}\mathbb{X}^{*}(t)\\
&\quad+\mathbb{A}(\alpha_{t})^{\top}\eta(t)+\mathbb{C}(\alpha_{t})^{\top}\zeta(t)+\mathbb{C}(\alpha_{t})^{\top}\mathbb{P}(\alpha_{t})\mathbb{I}_{n}\sigma(t)+q(t)+\mathbb{P}(\alpha_{t})\mathbb{I}_{n}b(t)+\beta(t)\\
&\quad -\big[\mathbb{P}(\alpha_{t})\mathbb{B}(\alpha_{t})+\mathbb{C}(\alpha_{t})^{\top}\mathbb{P}(\alpha_{t})\mathbb{D}(\alpha_{t})+\mathbb{S}(\alpha_{t})^{\top}\big]\mathbb{I}_{m}\Sigma(\mathbb{P},\alpha_{t})^{-1}\mathbb{J}^{\top}\\
&\quad\times\big[\mathbb{B}(\alpha_{t})^{\top}\eta(t)+ \mathbb{D}(\alpha_{t})^{\top}\zeta(t)
 +\mathbb{D}(\alpha_{t})^{\top}\mathbb{P}(\alpha_{t})\mathbb{I}_{n}\sigma(t)+\rho(t)\big],
\end{align*}
which implies that
\begin{equation}\label{decouple-P}
   \begin{aligned}
   0&=\mathbb{A}(i)^{\top}\mathbb{P}(i)+\mathbb{P}(i)\mathbb{A}(i)+\mathbb{C}(i)^{\top}\mathbb{P}(i)\mathbb{C}(i)+\mathbb{Q}(i)+\sum_{j=1}^{L}\pi_{ij}\mathbb{P}(j)\\
&\quad-\big[\mathbb{P}(i)\mathbb{B}(i)+\mathbb{C}(i)^{\top}\mathbb{P}(i)\mathbb{D}(i)+\mathbb{S}(i)^{\top}\big]\mathbb{I}_{m}\Sigma(\mathbb{P},i)^{-1}\mathbb{J}^{\top}\\
&\quad\times\big[\mathbb{B}(i)^{\top}\mathbb{P}(i)+\mathbb{D}(i)^{\top}\mathbb{P}(i)\mathbb{C}(i)+\mathbb{S}(i)
\big],\quad i\in\mathcal{S},
   \end{aligned}
\end{equation}
and
\begin{equation}\label{decouple-eta}
   \begin{aligned}
   d\eta(t)&=-\Big\{\mathbb{A}(\alpha_{t})^{\top}\eta(t)+\mathbb{C}(\alpha_{t})^{\top}\zeta(t)+\mathbb{C}(\alpha_{t})^{\top}\mathbb{P}(\alpha_{t})\mathbb{I}_{n}\sigma(t)+q(t)+\mathbb{P}(\alpha_{t})\mathbb{I}_{n}b(t)\\
&\quad -\big[\mathbb{P}(\alpha_{t})\mathbb{B}(\alpha_{t})+\mathbb{C}(\alpha_{t})^{\top}\mathbb{P}(\alpha_{t})\mathbb{D}(\alpha_{t})+\mathbb{S}(\alpha_{t})^{\top}\big]\mathbb{I}_{m}\Sigma(\mathbb{P},\alpha_{t})^{-1}\mathbb{J}^{\top}\\
&\quad\times\big[\mathbb{B}(\alpha_{t})^{\top}\eta(t)+ \mathbb{D}(\alpha_{t})^{\top}\zeta(t)+\mathbb{D}(\alpha_{t})^{\top}\mathbb{P}(\alpha_{t})\mathbb{I}_{n}\sigma(t)+\rho(t)\big]\Big\}dt\\
 &\quad+\zeta(t)dW(t)+\mathbf{z}(t)\cdot d\mathbf{\widetilde{N}}(t),\quad t\geq 0.
   \end{aligned}
\end{equation}
To sum up, we have the following result.
\begin{theorem}\label{thm-GLQ-open-closed}
Suppose that systems $[A,C;B_{1},D_{1}]_{\alpha}$ and $[A,C;B_{2},D_{2}]_{\alpha}$ are $L^{2}$-stablizable. If the following conditions hold:
\begin{description}
\item[(i)] The convexity condition \eqref{GLQ-convexity} holds;
  \item[(ii)] The CAREs \eqref{decouple-P}  admits a solution $\mathbb{P}\in\mathcal{D}\left(\mathbb{R}^{2n\times 2n}\right)$ such that $\Sigma(\mathbb{P},i)$ defined in \eqref{decouple-Sigma} is invertible and $\mathbf{\Theta}^{*}=\left[\Theta^{*}(1),\Theta^{*}(2),...,\Theta^{*}(L)\right]\in\mathcal{H}[A,C;B,D]_{\alpha}$, where
  \begin{equation}\label{GLQ-closedR-1}
       \Theta^{*}(i)\equiv(\Theta_{1}^{*}(i)^{\top},\Theta_{2}^{*}(i)^{\top})^{\top}\triangleq-\Sigma(\mathbb{P},i)^{-1}\mathbb{J}^{\top}\big[\mathbb{B}(i)^{\top}\mathbb{P}(i)+ \mathbb{D}(i)^{\top}\mathbb{P}(i)\mathbb{C}(i)+\mathbb{S}(i)\big]\mathbb{I}_{n};
   \end{equation}
  \item[(iii)] The BSDE \eqref{decouple-eta} admits an adapted solution $(\eta,\zeta,\mathbf{z})\in L_{\mathbb{F}}^{2}(\mathbb{R}^{2n})\times L_{\mathbb{F}}^{2}(\mathbb{R}^{2n})\times \mathcal{D}\left(L_{\mathcal{P}}^{2}(\mathbb{R}^{2n})\right)$.
\end{description}
Then for any $(x,i)\in\mathbb{R}^{n}\times\mathcal{S}$, the optimality system \eqref{FBSDEs-GLQ-2}-\eqref{stationary-GLQ-2} admits a solution:
\begin{equation}\label{optimality-system-solution-GLQ}
 \left\{
 \begin{aligned}
 &\mathbb{X}^{*}(\cdot;x,i)=\mathbb{I}_{n}X(\cdot;x,i,\mathbf{\Theta}^{*},\nu^{*}),\\
 &\mathbb{Y}^{*}(\cdot;x,i)=\mathbb{P}(\alpha)\mathbb{X}^{*}(\cdot;x,i)+\eta,\\
 &\mathbb{Z}^{*}(\cdot;x,i)=\mathbb{P}(\alpha)\left[\mathbb{C}(\alpha)\mathbb{X}^{*}(\cdot;x,i)+\mathbb{D}(\alpha)\mathbb{I}_{m}u^{*}+\mathbb{I}_{n}\sigma\right]+\zeta,\\
 &\Gamma_{j}^{*}(t;x,i)=\left[\mathbb{P}(j)-\mathbb{P}(\alpha(t-))\right]\mathbb{X}^{*}(t;x,i)+z_{j}(t),\quad j\in\mathcal{S}\quad t\geq 0,\\
 &u^{*}(\cdot;x,i)=\Theta^{*}(\alpha)X(\cdot;x,i,\mathbf{\Theta}^{*},\nu^{*})+\nu^{*},
 \end{aligned}
 \right.
\end{equation}
where
\begin{equation}\label{GLQ-closedR-2}
  \begin{aligned}
  \nu^{*}&\triangleq -\Sigma(\mathbb{P},\alpha)^{-1}\mathbb{J}^{\top}\big[\mathbb{B}(\alpha)^{\top}\eta + \mathbb{D}(\alpha)^{\top}\zeta+ \mathbb{D}(\alpha)^{\top}\mathbb{P}(\alpha)\mathbb{I}_{n}\sigma+\rho\big]\in L_{\mathbb{F}}^{2}(\mathbb{R}^{m}).
  \end{aligned}
  \end{equation}
In addition, the Problem (M-GLQ) is open-loop solvable and $(\mathbf{\Theta}^{*},\nu^{*})$ is the corresponding closed-loop representation strategy.
\end{theorem}

For $k=1,2$, if  $b=\sigma=q^{k}=0$, $\rho_{1}^{k}=0$, $\rho_{2}^{k}=0$, then the optimality system of Problem (M-GLQ)$^{0}$ becomes
 \begin{equation}\label{FBSDEs-GLQ-0}
  \left\{
      \begin{aligned}
      d\mathbb{X}^{*}(t)&=\left[\mathbb{A}(\alpha_{t})\mathbb{X}^{*}(t)+\mathbb{B}(\alpha_{t})\mathbb{I}_{m}u^{*}(t)\right]dt
      +\left[\mathbb{C}(\alpha_{t})\mathbb{X}^{*}(t)+\mathbb{D}(\alpha_{t})\mathbb{I}_{m}u^{*}(t)\right]dW(t),\\
      d\mathbb{Y}^{*}(t)&=-\left[\mathbb{A}(\alpha_{t})^{\top}\mathbb{Y}^{*}(t)+\mathbb{C}(\alpha_{t})^{\top}\mathbb{Z}^{*}(t)+\mathbb{Q}(\alpha_{t})\mathbb{X}^{*}(t)+\mathbb{S}(\alpha_{t})^{\top}\mathbb{I}_{m}u^{*}(t)\right]dt\\
      &\quad+\mathbb{Z}^{*}(t)dW(t)+\mathbf{\Gamma}^{*}(t)\cdot d\mathbf{\widetilde{N}}(t),\quad t\geq 0,\\
      \mathbb{X}^{*}(0)&=\mathbb{I}_{n}x,\quad\alpha_{0}=i,
      \end{aligned}
      \right.
  \end{equation}
with the stationary condition:
\begin{equation}\label{stationary-GLQ-0}
     \mathbb{J}^{\top}\left\{\mathbb{B}(\alpha_{t})^{\top}\mathbb{Y}^{*}(t)+ \mathbb{D}(\alpha_{t})^{\top}\mathbb{Z}^{*}(t)+\mathbb{S}(\alpha_{t})\mathbb{X}^{*}(t)+\mathbb{R}(\alpha_{t})\mathbb{I}_{m}u^{*}(t)\right\}=0,\quad a.e.\quad a.s..
  \end{equation}
One can easily verify that the above optimality system can be decoupled by $\mathbb{Y}^{*}=\mathbb{P}(\alpha)\mathbb{X}^{*}$. Therefore, we have the following result.

\begin{corollary}\label{coro-GLQ-open-solvability-0}
Suppose that systems $[A,C;B_{1},D_{1}]_{\alpha}$ and $[A,C;B_{2},D_{2}]_{\alpha}$ are $L^{2}$-stablizable. If the following conditions hold:
\begin{description}
\item[(i)] The convexity condition \eqref{GLQ-convexity} holds;
  \item[(ii)] The CAREs \eqref{decouple-P}  admits a solution $\mathbb{P}\in\mathcal{D}\left(\mathbb{R}^{2n\times 2n}\right)$ such that $\Sigma(\mathbb{P},i)$ defined in \eqref{decouple-Sigma} is invertible and corresponding $\mathbf{\Theta}^{*}$ defined in \eqref{GLQ-closedR-1} is in $\mathcal{H}[A,C;B,D]_{\alpha}$.
\end{description}
Then for any $(x,i)\in\mathbb{R}^{n}\times\mathcal{S}$, the optimality system \eqref{FBSDEs-GLQ-0}-\eqref{stationary-GLQ-0} admits a solution:
\begin{equation}\label{optimality-system-solution-GLQ-0}
 \left\{
 \begin{aligned}
 &\mathbb{X}^{*}(\cdot;x,i)=\mathbb{I}_{n}X^{0}(\cdot;x,i,\mathbf{\Theta}^{*},0),\\
 &\mathbb{Y}^{*}(\cdot;x,i)=\mathbb{P}(\alpha)\mathbb{X}^{*}(\cdot;x,i),\\
 &\mathbb{Z}^{*}(\cdot;x,i)=\mathbb{P}(\alpha)\left[\mathbb{C}(\alpha)\mathbb{X}^{*}(\cdot;x,i)+\mathbb{D}(\alpha)\mathbb{I}_{m}u^{*}\right],\\
 &\Gamma_{j}^{*}(t;x,i)=\left[\mathbb{P}(j)-\mathbb{P}(\alpha(t-))\right]\mathbb{X}^{*}(t;x,i),\quad j\in\mathcal{S}\quad t\geq 0,\\
 &u^{*}(\cdot;x,i)=\Theta^{*}(\alpha)X^{0}(\cdot;x,i,\mathbf{\Theta}^{*},0);
 \end{aligned}
 \right.
\end{equation}
 In addition, the Problem (M-GLQ)$^{0}$ is open-loop  solvable and $(\mathbf{\Theta}^{*},0)$  is the corresponding closed-loop representation strategy.
\end{corollary}

\begin{remark}\label{rmk-M-GLQ}
Note that $\mathbb{I}_{m}\Sigma(\mathbb{P},i)^{-1}\mathbb{J}^{\top}$ is not symmetric even if $\mathbb{P}$ is symmetric. Hence, the solution to CAREs \eqref{decouple-P} is generally not symmetric.  One can rewrite  \eqref{decouple-P} following from \eqref{decouple-P} and \eqref{GLQ-closedR-1} as:
\begin{equation}\label{decouple-P-2}
   \begin{aligned}
   0&=\Big[\mathbb{A}(i)^{\top}\mathbb{P}(i)+\mathbb{P}(i)\mathbb{A}(i)+\mathbb{C}(i)^{\top}\mathbb{P}(i)\mathbb{C}(i)+\mathbb{Q}(i)+\sum_{j=1}^{L}\pi_{ij}\mathbb{P}(j)\big]\mathbb{I}_{n}\\
&\quad+\big[\mathbb{P}(i)\mathbb{B}(i)+\mathbb{C}(i)^{\top}\mathbb{P}(i)\mathbb{D}(i)+\mathbb{S}(i)^{\top}\big]\mathbb{I}_{m}\Theta^{*}(i),\quad i\in\mathcal{S},
   \end{aligned}
\end{equation}
with constraint
\begin{equation}\label{decouple-P-constraint-2}
  \Sigma(\mathbb{P},i)\Theta^{*}(i)+\mathbb{J}^{\top}\big[\mathbb{B}(i)^{\top}\mathbb{P}(i)+ \mathbb{D}(i)^{\top}\mathbb{P}(i)\mathbb{C}(i)+\mathbb{S}(i)\big]\mathbb{I}_{n}=0.
\end{equation}
Substituting \eqref{GLQ-open-notation-2} into the above equations, we can further simplify \eqref{decouple-P-2}-\eqref{decouple-P-constraint-2} in the component forms:
\begin{equation}\label{decouple-P-3}
 \begin{aligned}
0&=A(i)^{\top}P_{k}(i)+P_{k}(i)A(i)+C(i)^{\top}P_{k}(i)C(i)+Q^{k}(i)+\sum_{j=1}^{L} \pi_{ij}P_{k}(j)\\
&\quad+\big[P_{k}(i)B(i)+C(i)^{\top}P_{k}(i)D(i)+S^{k}(i)^{\top}\big]\Theta^{*}(i),\quad i\in\mathcal{S},
   \end{aligned}
\end{equation}
with constraint
\begin{equation}\label{decouple-P-constraint-3}
\left(R_{k}^{k}(i)+D_{k}(i)^{\top}P_{k}(i)D(i)\right)\Theta^{*}(i)+B_{k}(i)^{\top}P_{k}(i)+ D_{k}(i)^{\top}P_{k}(i)C(i)+S_{k}^{k}(i)=0, \quad k=1,2.
\end{equation}
\end{remark}

 \subsection{Closed-loop solvability for Problem (M-GLQ)}\label{subsection-GLQ-closed}
 We now return to consider the closed-loop solvability for Problem (M-GLQ). 
Let
$$\widehat{X}_{1}=X(\cdot;x,i,u_{1},\mathbf{\widehat{\Theta}}_{2},\widehat{\nu}_{2}),\quad \widehat{X}_{2}=X(\cdot;x,i,\mathbf{\widehat{\Theta}}_{1},\widehat{\nu}_{1},u_{2}),$$
\begin{equation}\label{cost-GLQ-k}
\begin{aligned}
    \widehat{J}_{k}\left(x,i;u_{k}\right)
    & = \mathbb{E}\int_{0}^{\infty}\left[
    \left<
    \left(
    \begin{matrix}
    \widehat{Q}^{k}(\alpha) & \widehat{S}^{k}(\alpha)^{\top} \\
    \widehat{S}^{k}(\alpha) & R_{kk}^{k}(\alpha)
    \end{matrix}
    \right)
    \left(
    \begin{matrix}
    \widehat{X}_{k} \\
    u_{k} 
    \end{matrix}
    \right),
    \left(
    \begin{matrix}
    \widehat{X}_{k} \\
    u_{k}
    \end{matrix}
    \right)
    \right>
    +2\left<
    \left(
    \begin{matrix}
    \widehat{q}^{k} \\
    \widehat{\rho}^{k}
    \end{matrix}
    \right),
    \left(
    \begin{matrix}
    \widehat{X}_{k}\\
    u_{k}
    \end{matrix}
    \right)
    \right>\right]dt,\,k=1,2,
  \end{aligned}
\end{equation}
where
\begin{equation}\label{notation-GLQ}
    \left\{
\begin{aligned}
&\widehat{Q}^{1}(\alpha_{t})= Q^{1}(\alpha_{t})+S_{2}^{1}(\alpha_{t})^{\top}\widehat{\Theta}_{2}(\alpha_{t})
+\widehat{\Theta}_{2}(\alpha_{t})^{\top}S_{2}^{1}(\alpha_{t})
+\widehat{\Theta}_{2}(\alpha_{t})^{\top}R_{22}^{1}(\alpha_{t})\widehat{\Theta}_{2}(\alpha_{t}),\\
&\widehat{S}^{1}(\alpha_{t})=S_{1}^{1}(\alpha_{t})+R_{12}^{1}(\alpha_{t})\widehat{\Theta}_{2}(\alpha_{t}),\quad
\widehat{\rho}^{1}(t)=\rho_{1}^{1}(t)+R_{12}^{1}(\alpha_{t})\widehat{\nu}_{2}(t),\\
&\widehat{q}^{1}(t)=q^{1}(t)+\widehat{\Theta}_{2}(\alpha_{t})^{\top}\rho_{2}^{1}(t)
+\big[S_{2}^{1}(\alpha_{t})^{\top}+\widehat{\Theta}_{2}(\alpha_{t})^{\top}R_{22}^{1}(\alpha_{t})\big]\widehat{\nu}_{2}(t),\\
&\widehat{Q}^{2}(\alpha_{t})= Q^{2}(\alpha_{t})+S_{1}^{2}(\alpha_{t})^{\top}\widehat{\Theta}_{1}(\alpha_{t})
+\widehat{\Theta}_{1}(\alpha_{t})^{\top}S_{1}^{2}(\alpha_{t})
+\widehat{\Theta}_{1}(\alpha_{t})^{\top}R_{11}^{2}(\alpha_{t})\widehat{\Theta}_{1}(\alpha_{t}),\\
&\widehat{S}^{2}(\alpha_{t})=S_{2}^{2}(\alpha_{t})+R_{21}^{2}(\alpha_{t})\widehat{\Theta}_{1}(\alpha_{t}),\quad
\widehat{\rho}^{2}(t)=\rho_{2}^{2}(t)+R_{21}^{2}(\alpha_{t})\widehat{\nu}_{1}(t),\\
&\widehat{q}^{2}(t)=q^{2}(t)+\widehat{\Theta}_{1}(\alpha_{t})^{\top}\rho_{1}^{2}(t)
+\big[S_{1}^{2}(\alpha_{t})^{\top}+\widehat{\Theta}_{1}(\alpha_{t})^{\top}R_{11}^{2}(\alpha_{t})\big]\widehat{\nu}_{1}(t).
\end{aligned}
\right.
\end{equation}
Clearly, by some straightforward calculations, one has:
\begin{equation}\label{cost-GLQ}
    \begin{aligned}
        & J_{1}\left(x,i;u_{1},\mathbf{\widehat{\Theta}}_{2},\widehat{\nu}_{2}\right)
        =\widehat{J}_{1}\left(x,i;u_{1}\right)
        +\mathbb{E}\int_{0}^{\infty}\big<R_{22}^{1}(\alpha)\widehat{\nu}_{2}+2\rho_{2}^{1},\widehat{\nu}_{2}\big>dt,\\
        &J_{2}\left(x,i;\mathbf{\widehat{\Theta}}_{1},\widehat{\nu}_{1},u_{2}\right)
        =\widehat{J}_{2}\left(x,i;u_{2}\right)
        +\mathbb{E}\int_{0}^{\infty}\big<R_{11}^{2}(\alpha)\widehat{\nu}_{1}+2\rho_{1}^{2},\widehat{\nu}_{1}\big>dt,
    \end{aligned}
\end{equation}
 Then, by Definition \ref{def-closed-loop-equilibrium}, a 4-tuple $(\mathbf{\widehat{\Theta}_{1}},\widehat{\nu}_{1};\mathbf{\widehat{\Theta}_{2}},\widehat{\nu}_{2})$ is a closed-loop Nash equilibrium strategy of Problem (M-GLQ) if and only if the following holds:
\begin{description}
  \item[(i)] $\mathbf{\widehat{\Theta}}= (\mathbf{\widehat{\Theta}_{1}}^{\top},\mathbf{\widehat{\Theta}_{2}}^{\top})^{\top}\in \mathcal{H}\left[A,C;B,D\right]_{\alpha}$,
  \item[(ii)]  For $k=1,2$, $(\mathbf{\widehat{\Theta}_{k}},\widehat{\nu}_{k})$ is a closed-loop optimal control for LQ control problem with cost functional  $\widehat{J}_{k}\left(x,i;u_{k}\right)$ and state constraint $\widehat{X}_{k}$. 
\end{description}

The following theorem provides a characterization of the closed-loop Nash equilibrium strategy for Problem (M-GLQ).
\begin{theorem}\label{thm-GLQ-closed}
 A 4-tuple $(\mathbf{\widehat{\Theta}_{1}},\widehat{\nu}_{1};\mathbf{\widehat{\Theta}_{2}},\widehat{\nu}_{2})\in\mathcal{D}\left(\mathbb{R}^{m_{1}\times n}\right)\times L_{\mathbb{F}}^{2}(\mathbb{R}^{m_{1}})\times \mathcal{D}\left(\mathbb{R}^{m_{2}\times n}\right)\times L_{\mathbb{F}}^{2}(\mathbb{R}^{m_{2}})$ is a closed-loop Nash equilibrium strategy of Problem (M-GLQ) if and only if:
\begin{description}
  \item[(i)] $\mathbf{\widehat{\Theta}}= (\mathbf{\widehat{\Theta}_{1}}^{\top},\mathbf{\widehat{\Theta}_{2}}^{\top})^{\top}\in \mathcal{H}\left[A,C;B,D\right]_{\alpha}$,
  \item[(ii)]  The cross-coupled CAREs:
  \begin{equation}\label{CAREs-GLQ-1}
  \begin{aligned}
  0&=\mathcal{M}_{1}(P_{1},i)-\widehat{\Theta}_{1}(i)^{\top}\mathcal{N}_{11}^{1}(P_{1},i)\widehat{\Theta}_{1}(i)+\widehat{\Theta}_{2}(i)^{\top}\mathcal{N}_{22}^{1}(P_{1},i)\widehat{\Theta}_{2}(i)\\
  &\quad+\mathcal{L}_{2}^{1}(P_{1},i)\widehat{\Theta}_{2}(i)+\widehat{\Theta}_{2}(i)^{\top}\mathcal{L}_{2}^{1}(P_{1},i)^{\top},
  \end{aligned}
  \end{equation}
   \begin{equation}\label{CAREs-GLQ-2}
    \begin{aligned}
    0&=\mathcal{M}_{2}(P_{2},i)-\widehat{\Theta}_{2}(i)^{\top}\mathcal{N}_{22}^{2}(P_{2},i)\widehat{\Theta}_{2}(i)+\widehat{\Theta}_{1}(i)^{\top}\mathcal{N}_{11}^{2}(P_{2},i)\widehat{\Theta}_{1}(i)\\
  &\quad+\mathcal{L}_{1}^{2}(P_{2},i)\widehat{\Theta}_{1}(i)+\widehat{\Theta}_{1}(i)^{\top}\mathcal{L}_{1}^{2}(P_{2},i)^{\top},
  \end{aligned}
  \end{equation}
  admits a solution $(\mathbf{P_{1}},\mathbf{P_{2}})\in \mathcal{D}(\mathbb{S}^{n})\times  \mathcal{D}(\mathbb{S}^{n})$ such that, for any $i\in\mathcal{S}$,
  \begin{equation}\label{CAREs-GLQ-constraint}
  \left\{
      \begin{aligned}
      &\mathcal{N}_{11}^{1}(P_{1},i)\widehat{\Theta}_{1}(i)+\mathcal{N}_{12}^{1}(P_{1},i)\widehat{\Theta}_{2}(i)+\mathcal{L}_{1}^{1}(P_{1},i)^{\top}=0,\\
      &\mathcal{N}_{21}^{2}(P_{2},i)\widehat{\Theta}_{1}(i)+\mathcal{N}_{22}^{2}(P_{2},i)\widehat{\Theta}_{2}(i)+\mathcal{L}_{2}^{2}(P_{2},i)^{\top}=0,\\
      &\mathcal{N}_{11}^{1}(P_{1},i)\geq 0,\quad \mathcal{N}_{22}^{2}(P_{2},i)\geq 0,
      \end{aligned}
      \right.
  \end{equation}
  where, for any $k,l,m\in\{1,2\}$,
  $$
  \left\{
\begin{array}{l}
\mathcal{N}_{lm}^{k}(P_{k},i)=D_{l}(i)^{\top}P_{k}(i)D_{m}(i)+R_{lm}^{k}(i),\\
\mathcal{L}_{l}^{k}(P_{k},i)=P_{k}(i)B_{l}(i)+C(i)^{\top}P_{k}(i)D_{l}(i)+S_{l}^{k}(i)^{\top},\\
\mathcal{M}_{k}(P_{k},i)
= P_{k}(i)A(i)+A(i)^{\top}P_{k}(i)+C(i)^{\top}P_{k}(i)C(i)+Q^{k}(i)+\sum_{j=1}^{L}\pi_{ij}P_{k}(j),
\end{array}
\right.
$$
  \item[(iii)] The cross-coupled BSDEs:
  \begin{equation}\label{GLQ-eta-1}
   \begin{aligned}
   d\eta_{1}&=-\Big\{A(\alpha)^{\top}\eta_{1}+C(\alpha)^{\top}\big[\zeta_{1}+P_{1}(\alpha)\sigma\big]+P_{1}(\alpha)b+q^{1}+\widehat{\Theta}_{2}(\alpha)^{\top}\widehat{\rho}_{2}^{1}+\mathcal{L}_{2}^{1}(P_{1},\alpha)\widehat{\nu}_{2}\\
   &\quad-\widehat{\Theta}_{1}(\alpha)^{\top}\mathcal{N}_{11}^{1}(P_{1},\alpha)\widehat{\nu}_{1}+\widehat{\Theta}_{2}(\alpha)^{\top}\mathcal{N}_{22}^{1}(P_{1},\alpha)\widehat{\nu}_{2}\Big\}dt+\zeta_{1}dW(t)+\mathbf{z}_{1}\cdot d\mathbf{\widetilde{N}}(t),
   \end{aligned}
  \end{equation}
  and
\begin{equation}\label{GLQ-eta-2}
   \begin{aligned}
   d\eta_{2}&=-\Big\{A(\alpha)^{\top}\eta_{2}+C(\alpha)^{\top}\big[\zeta_{2}+P_{2}(\alpha)\sigma\big]+P_{2}(\alpha)b+q^{2}+\widehat{\Theta}_{1}(\alpha)^{\top}\widehat{\rho}_{1}^{2}+\mathcal{L}_{1}^{2}(P_{2},\alpha)\widehat{\nu}_{1}\\
   &\quad+\widehat{\Theta}_{1}(\alpha)^{\top}\mathcal{N}_{11}^{2}(P_{2},\alpha)\widehat{\nu}_{1}-\widehat{\Theta}_{2}(\alpha)^{\top}\mathcal{N}_{22}^{2}(P_{2},\alpha)\widehat{\nu}_{2}\Big\}dt+\zeta_{2}dW(t)+\mathbf{z}_{2}\cdot d\mathbf{\widetilde{N}}(t),
   \end{aligned}
  \end{equation}
  admits a solution $(\eta_{1},\zeta_{1},\mathbf{z}_{1};\eta_{2},\zeta_{2},\mathbf{z}_{2})\in L_{\mathbb{F}}^{2}(\mathbb{R}^{n})\times L_{\mathbb{F}}^{2}(\mathbb{R}^{n})\times\mathcal{D}\left(L_{\mathcal{P}}^{2}(\mathbb{R}^{n})\right)\times L_{\mathbb{F}}^{2}(\mathbb{R}^{n})\times L_{\mathbb{F}}^{2}(\mathbb{R}^{n})\times\mathcal{D}\left(L_{\mathcal{P}}^{2}(\mathbb{R}^{n})\right)$ such that
  \begin{equation}\label{CAREs-GLQ-eta-constraint}
  \left\{
      \begin{aligned}
      &\mathcal{N}_{11}^{1}(P_{1},\alpha)\widehat{\nu}_{1}+\mathcal{N}_{12}^{1}(P_{1},\alpha)\widehat{\nu}_{2}+\widehat{\rho}_{1}^{1}=0,\\
      &\mathcal{N}_{21}^{2}(P_{2},\alpha)\widehat{\nu}_{1}+\mathcal{N}_{22}^{2}(P_{2},\alpha)\widehat{\nu}_{2}+\widehat{\rho}_{2}^{2}=0,
      \end{aligned}
      \right.
  \end{equation}
  where
 $$
\widehat{\rho}_{l}^{k}\triangleq B_{l}(\alpha)^{\top}\eta_{k}+D_{l}(\alpha)^{\top}\zeta_{k}
+D_{l}(\alpha)^{\top}P_{k}(\alpha)\sigma+\rho_{l}^{k},\quad k,l\in\{1,2\}.
$$
\end{description}
  In this case, the closed-loop equilibrium value function of  two players are given by
  \begin{equation}\label{GLQ-value-1}
  \begin{aligned}
  V_{1}(x,i)&=\big<P_{1}(i)x,x\big>+\mathbb{E}\Big\{2\big<\eta_{1}(0),x\big>+\int_{0}^{\infty}\Big[2\big<\eta_{1},b\big>+2\big<\zeta_{1},\sigma\big>+\big<P_{1}(\alpha)\sigma,\sigma\big>\\
  &\quad +2\big<\widehat{\rho}_{2}^{1}-\mathcal{N}_{21}^{1}(P_{1},\alpha)\mathcal{N}_{11}^{1}(P_{1},\alpha)^{\dag}\widehat{\rho}_{1}^{1},\widehat{\nu}_{2}\big>-\big<\mathcal{N}_{11}^{1}(P_{1},\alpha)^{\dag}\widehat{\rho}_{1}^{1},\widehat{\rho}_{1}^{1}\big>\\
  &\quad +\big<\big[\mathcal{N}_{22}^{1}(P_{1},\alpha)-\mathcal{N}_{21}^{1}(P_{1},\alpha)\mathcal{N}_{11}^{1}(P_{1},\alpha)^{\dag}\mathcal{N}_{12}^{1}(P_{1},\alpha)\big]\widehat{\nu}_{2},\widehat{\nu}_{2}\big>
  \Big]dt\Big\},
  \end{aligned}
  \end{equation}
  and
   \begin{equation}\label{GLQ-value-2}
  \begin{aligned}
  V_{2}(x,i)&=\big<P_{2}(i)x,x\big>+\mathbb{E}\Big\{2\big<\eta_{2}(0),x\big>+\int_{0}^{\infty}\Big[2\big<\eta_{2},b\big>+2\big<\zeta_{2},\sigma\big>+\big<P_{2}(\alpha)\sigma,\sigma\big>\\
  &\quad +2\big<\widehat{\rho}_{1}^{2}-\mathcal{N}_{12}^{2}(P_{2},\alpha)\mathcal{N}_{22}^{2}(P_{2},\alpha)^{\dag}\widehat{\rho}_{2}^{2},\widehat{\nu}_{1}\big>-\big<\mathcal{N}_{22}^{2}(P_{2},\alpha)^{\dag}\widehat{\rho}_{2}^{2},\widehat{\rho}_{2}^{2}\big>\\
  &\quad +\big<\big[\mathcal{N}_{11}^{2}(P_{2},\alpha)-\mathcal{N}_{12}^{2}(P_{2},\alpha)\mathcal{N}_{22}^{2}(P_{2},\alpha)^{\dag}\mathcal{N}_{21}^{2}(P_{2},\alpha)\big]\widehat{\nu}_{1},\widehat{\nu}_{1}\big>
  \Big]dt\Big\}.
  \end{aligned}
  \end{equation}
\end{theorem}

\begin{proof}
For any given $\mathbf{P_{k}}\in\mathcal{D}\left(\mathbb{S}^{n}\right),\, k=1,2$, let
$$
\left\{
\begin{array}{l}
\widehat{A}_{1}(i)\triangleq A(i)+B_{2}(i)\widehat{\Theta}_{2}(i),\quad
\widehat{A}_{2}(i)\triangleq A(i)+B_{1}(i)\widehat{\Theta}_{1}(i),\\
\widehat{C}_{1}(i)\triangleq C(i)+D_{2}(i)\widehat{\Theta}_{2}(i),\quad
\widehat{C}_{2}(i)\triangleq C(i)+D_{1}(i)\widehat{\Theta}_{1}(i),\\
\widehat{b}_{1}\triangleq B_{2}(\alpha)\widehat{\nu}_{2}+b,\quad \widehat{b}_{2}\triangleq B_{1}(\alpha)\widehat{\nu}_{1}+b,\quad
\widehat{\sigma}_{1}\triangleq D_{2}(\alpha)\widehat{\nu}_{2}+\sigma,\quad
\widehat{\sigma}_{2}\triangleq D_{1}(\alpha)\widehat{\nu}_{1}+\sigma,\\
\widehat{\mathcal{M}}_{k}(P_{k},i)\triangleq P_{k}(i)\widehat{A}_{k}(i)+\widehat{A}_{k}(i)^{\top}P_{k}(i)+\widehat{C}_{k}(i)^{\top}P_{k}(i)\widehat{C}_{k}(i)+\widehat{Q}^{k}(i)+\sum_{j=1}^{L}\pi_{ij}P_{k}(j),\\
\widehat{\mathcal{L}}_{k}(P_{k},i)\triangleq P_{k}(i)B_{k}(i)+\widehat{C}_{k}(i)^{\top}P_{k}(i)D_{k}(i)+\widehat{S}^{k}(i)^{\top},\\
\widehat{\mathcal{N}}_{k}(P_{k},i)\triangleq D_{k}(i)^{\top}P_{k}(i)D_{k}(i)+R_{kk}^{k}(i),\quad i\in\mathcal{S}.
\end{array}
\right.
$$
Then by \cite[Theorem $5.2$]{Wu-etal} and some basic properties of pseudoinverse (see Penrose \cite{Penrose.1955}), for $k=1,2$, we have that $(\mathbf{\widehat{\Theta}_{k}},\widehat{\nu}_{k})$ is an optimal closed-loop control for problem with cost functional $\widehat{J}_{k}\left(x,i;u_{k}\right)$  and state constraint $X_{k}$ if and only if the following holds:
\begin{description}
  \item[(i)] The constrained CAREs:
  \begin{equation}\label{CAREs-GLQ-K}
\left\{
   \begin{aligned}
    &\widehat{\mathcal{M}}_{k}(P_{k},i)-\widehat{\mathcal{L}}_{k}(P_{k},i) \widehat{\mathcal{N}}_{k}(P_{k},i)^{\dag} \widehat{\mathcal{L}}_{k}(P_{k},i)^{\top} = 0,\\
    & \widehat{\mathcal{N}}_{k}(P_{k},i)\geq 0,\quad\forall i\in\mathcal{S},
   \end{aligned}
   \right.
\end{equation}
admits a solution $\mathbf{P}_{k}$ such that
\begin{equation}\label{CAREs-GLQ-K-constraint}
    \widehat{\mathcal{N}}_{k}(P_{k},i)\widehat{\Theta}_{k}(i)+\widehat{\mathcal{L}}_{k}(P_{k},i)^{\top} = 0,
\end{equation}
\item[(ii)] The BSDE
\begin{equation}\label{GLQ-eta-k}
      \begin{aligned}
        d\eta_{k}&=-\big\{\big[\widehat{A}_{k}(\alpha)^{\top}-\widehat{\mathcal{L}}_{k}(P_{k},\alpha)\widehat{\mathcal{N}}_{k}(P_{k},\alpha)^{\dag}B_{k}(\alpha)^{\top}\big]\eta_{k}-\widehat{\mathcal{L}}_{k}(P_{k},\alpha)\widehat{\mathcal{N}}_{k}(P_{k},\alpha)^{\dag}\widehat{\rho}^{k}\\
        &\quad+\big[\widehat{C}_{k}(\alpha)^{\top}-\widehat{\mathcal{L}}_{k}(P_{k},\alpha)\widehat{\mathcal{N}}_{k}(P_{k},\alpha)^{\dag}D_{k}(\alpha)^{\top}\big]\left(\zeta_{k}+P_{k}(\alpha)\widehat{\sigma}_{k}\right)+P_{k}(\alpha)\widehat{b}_{k}+\widehat{q}^{k}\big\}dt\\
        &\quad+\zeta_{k} dW(t)+\mathbf{z}_{k}\cdot d\mathbf{\widetilde{N}}(t),\quad t\geq 0,
      \end{aligned}
  \end{equation}
  admits a $L^{2}$-stable adapted solution  $\left(\eta_{k},\zeta_{k},\mathbf{z}_{k}\right)$ such that
   \begin{equation}\label{GLQ-eta-k-constraint}
    \widehat{\mathcal{N}}_{k}(P_{k},\alpha_{t})\widehat{\nu}_{k}(t)+\widetilde{\rho}^{k}(t)=0,\text{ } a.e.\text{ } a.s..
   \end{equation}
   where
   $$\widetilde{\rho}^{k}=B_{k}(\alpha)^{\top}\eta_{k}+D_{k}(\alpha)^{\top}\zeta_{k}+D_{k}(\alpha)^{\top}P_{k}(\alpha)\widehat{\sigma}_{k}+\widehat{\rho}^{k}.$$
\end{description}
In this case, the value function $V_{k}(x,i)$ admits the following representation:
\begin{equation}\label{GLQ-value-function-K}
    \begin{aligned}
      V_{k}(x,i)&=\big<P_{k}(i)x,x\big>+\mathbb{E}\big\{2\big<\eta_{k}(0),x\big>+\int_{0}^{\infty}
      \big[\big<P_{k}(\alpha)\widehat{\sigma}_{k},\widehat{\sigma}_{k}\big>+2\big<\eta_{k},\widehat{b}_{k}\big>+2\big<\zeta_{k},\widehat{\sigma}_{k}\big>\\
      &\quad-\big<\widehat{\mathcal{N}}_{k}(P_{k},\alpha)^{\dag}\widetilde{\rho}^{k},\widetilde{\rho}^{k}\big>
      +\big<R_{ll}^{k}(\alpha)\widehat{\nu}_{l}+2\rho_{l}^{k},\widehat{\nu}_{l}\big>\big]dt\big\},\quad k,l\in\{(1,2),(2,1)\}.
    \end{aligned}
\end{equation}
On the other hand, one can easily verify that
\begin{equation}\label{GLQ-relation}
 \left\{
\begin{array}{l}
\widehat{\mathcal{M}}_{k}(P_{k},i)=\mathcal{M}_{k}(P_{k},i)+\mathcal{L}_{l}^{k}(P_{k},i)\widehat{\Theta}_{l}(i)
+\widehat{\Theta}_{l}(i)^{\top}\mathcal{L}_{l}^{k}(P_{k},i)^{\top}
+\widehat{\Theta}_{l}(i)^{\top}\mathcal{N}_{ll}^{k}(P_{k},i)\widehat{\Theta}_{l}(i),\\
\widehat{\mathcal{L}}_{k}(P_{k},i)=\mathcal{L}_{k}^{k}(P_{k},i)+\widehat{\Theta}_{l}(i)^{\top}\mathcal{N}_{lk}^{k}(P_{k},i),\\
\widehat{\mathcal{N}}_{k}(P_{k},i)=\mathcal{N}_{kk}^{k}(P_{k},i),\quad i\in\mathcal{S}, \quad (k,l)\in\{(1,2),(2,1)\}.
\end{array}
\right.
\end{equation}
Substituting the above equations into \eqref{CAREs-GLQ-K-constraint} and \eqref{GLQ-eta-k-constraint} yields \eqref{CAREs-GLQ-constraint} and \eqref{CAREs-GLQ-eta-constraint}. In addition, plugging \eqref{GLQ-relation} into \eqref{CAREs-GLQ-K}, we have
\begin{align*}
     0&=\mathcal{M}_{k}(P_{k},i)-\big[\widehat{\Theta}_{l}^{\top}\mathcal{N}_{lk}^{k}(P_{k},i)+\mathcal{L}_{k}^{k}(P_{k},i)\big]\mathcal{N}_{kk}^{k}(P_{k},i)^{\dag}\big[\mathcal{N}_{kl}^{1}(P_{k},i)\widehat{\Theta}_{l}
   +\mathcal{L}_{k}^{k}(P_{k},i)^{\top}\big]\\
&\quad+\mathcal{L}_{l}^{k}(P_{k},i)\widehat{\Theta}_{l}+\widehat{\Theta}_{l}^{\top}\mathcal{L}_{l}^{k}(P_{k},i)^{\top}+\widehat{\Theta}_{l}^{\top}\mathcal{N}_{ll}^{k}(P_{k},i)\widehat{\Theta}_{l},\quad i\in\mathcal{S}, \quad (k,l)\in\{(1,2),(2,1)\}.
\end{align*}
It follows from \eqref{CAREs-GLQ-constraint} that the above equation can be simplify as \eqref{CAREs-GLQ-1}-\eqref{CAREs-GLQ-2}. On the other hand, taking $k=1$, it follows from \eqref{GLQ-eta-k} that
\begin{equation}\label{GLQ-eta-k-1}
    \begin{aligned}
        d\eta_{1}
        &=-\big\{\big[A(\alpha)+B_{2}(\alpha)\widehat{\Theta}_{2}(\alpha)\big]^{\top}\eta_{1}+\big[C(\alpha)+D_{2}(\alpha)\widehat{\Theta}_{2}(\alpha)\big]^{\top}\big[\zeta_{1}+P_{1}(\alpha)\sigma+P_{1}(\alpha)D_{2}(\alpha)\widehat{\nu}_{2}\big]\\
        &\quad+P_{1}(\alpha)b+q^{1}+P_{1}(\alpha)B_{2}(\alpha)\widehat{\nu}_{2}+\widehat{\Theta}_{2}(\alpha)^{\top}\rho_{2}^{1}+\big[S_{2}^{1}(\alpha)^{\top}+\widehat{\Theta}_{2}(\alpha)^{\top}R_{22}^{1}(\alpha)\big]\widehat{\nu}_{2}\\
        &\quad-\big[\mathcal{L}_{1}^{1}(P_{1},\alpha)+\widehat{\Theta}_{2}(\alpha)^{\top}\mathcal{N}_{21}^{1}(P_{1},\alpha)\big]\widehat{\mathcal{N}}_{11}^{1}(P_{1},\alpha)^{\dag}\big[B_{1}(\alpha)^{\top}\eta_{1}+D_{1}(\alpha)^{\top}\big(\zeta_{1}+P_{1}(\alpha)\sigma
        \big)+\rho_{1}^{1}\\
        &\quad+D_{1}(\alpha)^{\top}P_{1}(\alpha)D_{2}(\alpha)\widehat{\nu}_{2}+R_{12}^{1}(\alpha)\widehat{\nu}_{2}\big]\big\}dt+\zeta_{1} dW(t)+\mathbf{z}_{1}\cdot d\mathbf{\widetilde{N}}(t),\\
        &=-\big\{A(\alpha)^{\top}\eta_{1}+C(\alpha)^{\top}\big[\zeta_{1}+P_{1}(\alpha)\sigma\big]+P_{1}(\alpha)b+q^{1}+\widehat{\Theta}_{2}(\alpha)^{\top}\widehat{\rho}_{2}^{1}+\mathcal{L}_{2}^{1}(P_{1},\alpha)\widehat{\nu}_{2}\\
   &\quad+\widehat{\Theta}_{2}(\alpha)^{\top}\mathcal{N}_{22}^{1}(P_{1},\alpha)\widehat{\nu}_{2}-\big[\mathcal{L}_{1}^{1}(P_{1},\alpha)+\widehat{\Theta}_{2}(\alpha)^{\top}\mathcal{N}_{21}^{1}(P_{1},\alpha)\big]\widehat{\mathcal{N}}_{11}^{1}(P_{1},\alpha)^{\dag}\big[\widehat{\rho}_{1}^{1}+\mathcal{N}_{12}^{1}(P,\alpha)\widehat{\nu}_{2}\big]\big\}\\
   &\quad+\zeta_{1}dW(t)+\mathbf{z}_{1}\cdot d\mathbf{\widetilde{N}}(t),\quad t\geq 0.
\end{aligned}
\end{equation}
It follows from \eqref{CAREs-GLQ-constraint} and \eqref{CAREs-GLQ-eta-constraint} that
\begin{align*}
   & \mathcal{L}_{1}^{1}(P_{1},\alpha)+\widehat{\Theta}_{2}(\alpha)^{\top}\mathcal{N}_{21}^{1}(P_{1},\alpha)=-\widehat{\Theta}_{1}(\alpha)^{\top}\mathcal{N}_{11}^{1}(P_{1},\alpha),\\
   &\widehat{\rho}_{1}^{1}+\mathcal{N}_{12}^{1}(P,\alpha)\widehat{\nu}_{2}=-\mathcal{N}_{11}^{1}(P,\alpha)\widehat{\nu}_{1}.
\end{align*}
Plugging the above relations into \eqref{GLQ-eta-k-1} yields \eqref{GLQ-eta-1} and one can derive \eqref{GLQ-eta-2} similarly. Finally, we can also obtain \eqref{GLQ-value-1}-\eqref{GLQ-value-2} by substituting \eqref{GLQ-relation} into \eqref{GLQ-value-function-K}. This completes the proof.
\end{proof}

\begin{remark}
In fact,  let
$$
\left\{
\begin{array}{l}
\mathcal{L}^{k}(P_{k},i)\triangleq P_{k}(i)B(i)+C(i)^{\top}P_{k}(i)D(i)+S^{k}(i)^{\top},\\
\mathcal{N}^{k}(P_{k},i)\triangleq D(i)^{\top}P_{k}(i)D(i)+R^{k}(i),\\
\mathcal{N}_{k}^{k}(P_{k},i)\triangleq D_{k}(i)^{\top}P_{k}(i)D(i)+R_{k}^{k}(i),\quad i\in\mathcal{S},\quad k=1,2.
\end{array}
\right.
$$
Then one can further simplify the CAREs\eqref{CAREs-GLQ-1}-\eqref{CAREs-GLQ-2} by using \eqref{CAREs-GLQ-constraint} as follows:
\begin{equation}\label{CAREs-GLQ-1-2}
\begin{aligned}
    0&=\mathcal{M}_{1}(P_{1},i)+\widehat{\Theta}^{\top}\mathcal{N}^{1}(P_{1},i)\widehat{\Theta}+\mathcal{L}^{1}(P_{1},i)\widehat{\Theta}+\widehat{\Theta}^{\top}\mathcal{L}^{1}(P_{1},i)^{\top},
\end{aligned}
\end{equation}
and
\begin{equation}\label{CAREs-GLQ-2-2}
\begin{aligned}
  0&=\mathcal{M}_{2}(P_{2},i)+\widehat{\Theta}^{\top}\mathcal{N}^{2}(P_{2},i)\widehat{\Theta}+\mathcal{L}^{2}(P_{2},i)\widehat{\Theta}+\widehat{\Theta}^{\top}\mathcal{L}^{2}(P_{2},i)^{\top}.
\end{aligned}
\end{equation}
In addition, the constraint condition \eqref{CAREs-GLQ-constraint} also admits the following representation:
  \begin{equation}\label{CAREs-GLQ-constraint-2}
  \left\{
      \begin{aligned}
      &\mathcal{N}_{1}^{1}(P_{1},i)\widehat{\Theta}+\mathcal{L}_{1}^{1}(P_{1},i)^{\top}=0,\\
      &\mathcal{N}_{2}^{2}(P_{2},i)\widehat{\Theta}+\mathcal{L}_{2}^{2}(P_{2},i)^{\top}=0,\\
      &\mathcal{N}_{11}^{1}(P_{1},i)\geq 0,\quad \mathcal{N}_{22}^{2}(P_{2},i)\geq 0.
      \end{aligned}
      \right.
  \end{equation}
  Therefore,  the constrained cross-coupled CAREs\eqref{CAREs-GLQ-1}-\eqref{CAREs-GLQ-constraint} can be rewritten as CAREs\eqref{CAREs-GLQ-1-2}-\eqref{CAREs-GLQ-constraint-2}.
 This is consistent with the corresponding result in Li et al. \cite{Li-Shi-Yong-2021-ID-MFLQ-IF}.
\end{remark}

For $k=1, 2$,  if $b=\sigma= q^{k}=0$, $\rho_{1}^{k}=0$, $\rho_{2}^{k}=0$, then the following result immediately follows from the sufficiency of Theorem \ref{thm-GLQ-closed}.
\begin{corollary}\label{coro-GLQ-0-closed}
If $(\mathbf{\widehat{\Theta}_{1}},\widehat{\nu}_{1};\mathbf{\widehat{\Theta}_{2}},\widehat{\nu}_{2})$ is a closed-loop Nash equilibrium strategy of Problem (M-GLQ), then $(\mathbf{\widehat{\Theta}_{1}},0;\mathbf{\widehat{\Theta}_{2}},0)$ is a closed-loop Nash equilibrium strategy of Problem (M-GLQ)$^{0}$. In this case, the closed-loop equilibrium value function of  two players are given by
  \begin{equation}\label{GLQ-0-value}
  V_{1}^{0}(x,i)=\big<P_{1}(i)x,x\big>,\quad
  V_{2}^{0}(x,i)=\big<P_{2}(i)x,x\big>.
  \end{equation}
\end{corollary}

\section{Zero-sum Nash differential game}\label{section-ZLQ}
In this section, we study the zero-sum LQ-SDG problem.  By definition, we know that the Problem (M-ZLQ) can be considered a special case of Problem (M-GLQ). Hence, we can use the results derived in the previous section to study the open-loop and closed-loop solvabilities of Problem (M-ZLQ).
\subsection{Open-loop solvability for Problem (M-ZLQ)}
It follows from \eqref{FBSDEs-GLQ}
 and the relation \eqref{ZLQ-cost-notation} that, for Problem (M-ZLQ), the following holds: $$\left(Y_{1}^{*},Z_{1}^{*},\mathbf{\Gamma}_{1}^{*}\right)=-\left(Y_{2}^{*},Z_{2}^{*},\mathbf{\Gamma}_{2}^{*}\right)\triangleq \left(Y^{*},Z^{*},\mathbf{\Gamma}^{*}\right).$$
Consequently, we immediately derive the open-loop solvability of Problem (M-ZLQ) from Theorem \ref{thm-GLQ-open-solvability} as follows.

\begin{theorem}\label{thm-ZLQ-open-solvability}
 Suppose that systems $[A,C;B_{1},D_{1}]_{\alpha}$ and $[A,C;B_{2},D_{2}]_{\alpha}$ are $L^{2}$-stabilizable. Then $(u_{1}^{*},u_{2}^{*})\in\mathcal{U}_{ad}(x,i)$ is an open-loop saddle point of Problem (M-ZLQ) for initial value $(x,i)\in\mathbb{R}^{n}\times\mathcal{S}$ if and only if
 \begin{description}
    \item[(i)] The following convexity-concavity condition holds:
   \begin{equation}\label{ZLQ-convexity-concavity}
       \left\{
       \begin{array}{l}
         J^{0}(0,i;u_{1},0)\geq 0,\quad \forall u_{1}\in\mathcal{U}_{ad}^{1,0}(x,i;0),\quad \forall i\in\mathcal{S},  \\
         J^{0}(0,i;0,u_{2})\leq 0,\quad \forall u_{2}\in\mathcal{U}_{ad}^{2,0}(x,i;0),  \quad \forall i\in\mathcal{S},
       \end{array}
       \right.
   \end{equation}
   \item[(ii)] The adapted solution $\left(X^{*},Y^{*},Z^{*},\mathbf{\Gamma}^{*}\right)\in L_{\mathbb{F}}^{2}(\mathbb{R}^{n})\times L_{\mathbb{F}}^{2}(\mathbb{R}^{n})\times L_{\mathbb{F}}^{2}(\mathbb{R}^{n})\times\mathcal{D}\left(L_{\mathcal{P}}^{2}(\mathbb{R}^{n})\right)$ to the  following FBSDEs
  \begin{equation}\label{FBSDEs-ZLQ}
  \left\{
      \begin{aligned}
      dX^{*}(t)&=\left[A(\alpha_{t})X^{*}(t)+B(\alpha_{t})u^{*}(t)+b(t)\right]dt+\left[C(\alpha_{t})X^{*}(t)+D(\alpha_{t})u^{*}(t)+\sigma(t)\right]dW(t),\\
      dY^{*}(t)&=-\left[A(\alpha_{t})^{\top}Y^{*}(t)+C(\alpha_{t})^{\top}Z^{*}(t)+Q(\alpha_{t})X^{*}(t)+S(\alpha_{t})^{\top}u^{*}(t)+q(t)\right]dt\\
      &\quad+Z^{*}(t)dW(t)+\mathbf{\Gamma}^{*}(t)\cdot d\mathbf{\widetilde{N}}(t),\quad t\geq 0,\\
      X^{*}(0)&=x,\quad\alpha_{0}=i,
      \end{aligned}
      \right.
  \end{equation}
  satisfies the following stationary condition:
  \begin{equation}\label{stationary-ZLQ}
     B(\alpha_{t})^{\top}Y^{*}(t)+ D(\alpha_{t})^{\top}Z^{*}(t)+S(\alpha_{t})X^{*}(t)+R(\alpha_{t})u^{*}(t)+\rho(t)=0,\quad a.e.\quad a.s..
  \end{equation}
 \end{description}
\end{theorem}

In the above, the FBSDEs \eqref{FBSDEs-ZLQ} together with the stationary condition \eqref{stationary-ZLQ} constitute the optimality system of Problem (M-ZLQ).

Note that for any $i\in\mathcal{S}$,
\[
\left(\begin{matrix}
Q(i) & S_{1}(i)^{\top}\\
S_{1}(i) & R_{11}(i)\end{matrix}\right)\geq 0,\quad \text{and} \quad
\left(\begin{matrix}
Q(i) & S_{2}(i)^{\top}\\
S_{2}(i) & R_{22}(i)\end{matrix}\right)\leq 0
\]
generally cannot hold together. Hence, a natural question arises: under what conditions does the performance functional \eqref{zero-sum-performance} satisfy condition \eqref{ZLQ-convexity-concavity}? In the following, we  say $\lambda$ ($\mu$) is the minimum (maximum) eigenvalue of process $\Lambda(\alpha)$ if it satisfies
$$\lambda=\min\{\lambda_{1},\lambda_{2},\cdots,\lambda_{L}\}\quad
\left(\mu=\max\{\mu_{1},\mu_{2},\cdots,\mu_{L}\}\right),$$
where $\lambda_{i}$ ($\mu_{i}$) is the  minimum (maximum) eigenvalue of $\Lambda(i)$,
$i\in\mathcal{S}$. For simplicity, we denote the minimum (maximum) eigenvalue of a process $\Lambda(\alpha)$ as $meig(\Lambda(\alpha))$  ($Meig(\Lambda(\alpha))$). Let
$$
\left\{
\begin{array}{l}
m_{Q}=meig(Q(\alpha)),\quad M_{Q}=Meig(Q(\alpha)), \quad m_{R}=meig(R_{11}(\alpha)),\quad M_{R}=Meig(R_{22}(\alpha)),\\
M(\epsilon)=Meig\left(A(\alpha)+A(\alpha)^{\top}+C(\alpha)^{\top}C(\alpha)+\epsilon I\right),\\
\mu_{k}(\epsilon)=Meig\big(\frac{1}{\epsilon}[B_{k}(\alpha)+C(\alpha)^{\top}D_{k}(\alpha)]^{\top}[B_{k}(\alpha)+C(\alpha)^{\top}D_{k}(\alpha)]+D_{k}(\alpha)^{\top}D_{k}(\alpha)\big).
\end{array}
\right.
$$
The following result provides a sufficient condition to verify that the condition \eqref{ZLQ-convexity-concavity} holds.
\begin{proposition}\label{prop-convex-concave}
Suppose $m_{Q}<0,\,M_{Q}>0$ and the following conditions hold:
\begin{enumerate}
  \item [(i)] $S_{k}(i),\, R_{12}(i)=R_{21}(i)=0$ for all $i\in\mathcal{S}$;
  \item [(ii)] There exists two constants $\epsilon_{1}>0$  and  $\epsilon_{2}>0$  such that
  $$M(\epsilon_{1})<0,\quad m_{R}-\frac{\mu_{1}(\epsilon_{1})m_{Q}}{M(\epsilon_{1})}\geq 0, \quad \text{and} \quad
  M(\epsilon_{2})<0,\quad M_{R}-\frac{\mu_{2}(\epsilon_{2})M_{Q}}{M(\epsilon_{2})}\leq 0.$$
\end{enumerate}
Then, the performance functional \eqref{zero-sum-performance} satisfies condition \eqref{ZLQ-convexity-concavity}.
\end{proposition}
\begin{proof}
  Let $X_{k}=X^{0}(\cdot;0,i,u_{k},0)$.
  By applying It\^o's rule, we have
  \[
  \begin{aligned}
  \mathbb{E}[|X_{k}(t)|^{2}]=\mathbb{E}\int_{0}^{t}\Big[&\left<\big(A(\alpha)+A(\alpha)^{\top}+C(\alpha)^{\top}C(\alpha)\big)X_{k},X_{k}\right>+\left<D_{k}(\alpha)^{\top}D_{k}(\alpha)u_{k},u_{k}\right>\\
  &+2\left<\big(B_{k}(\alpha)+C(\alpha)^{\top}D_{k}(\alpha)\big)u_{k},X_{k}\right>\Big]ds.
  \end{aligned}
  \]
  Set $\phi_{k}(t)=\mathbb{E}[|X_{k}(t)|^{2}]$ and $\psi_{k}(t)=\mathbb{E}[|u_{k}(t)|^{2}]$. Then the above equation implies
  \[
  d\phi_{k}(t)\leq \left[M(\epsilon_{k})\phi_{k}(t)+\mu_{k}(\epsilon_{k})\psi_{k}(t)\right]dt,\quad \forall \epsilon_{k}>0.
  \]
  Note that $\phi_{k}(0)=\mathbb{E}[|X_{k}(0)|^{2}]=0$ . By Gronwall's inequality, one can further obtain
  \[
  \phi_{k}(t)\leq \int_{0}^{t}\mu_{k}(\epsilon_{k})\psi_{k}(s)e^{M(\epsilon_{k})(t-s)}ds.
  \]
  Taking $\epsilon_{k}$ satisfied the condition $(ii)$ and integrating both side yields
  \[\begin{aligned}
 \mathbb{E}\int_{0}^{\infty}|X_{k}(t)|^{2}dt= \int_{0}^{\infty}\phi_{k}(t)dt&\leq  \int_{0}^{\infty}\big(\int_{0}^{t}\mu_{k}(\epsilon_{k})\psi_{k}(s)e^{M(\epsilon_{k})(t-s)}ds\big) dt\\
  &= \int_{0}^{\infty}\big(\int_{s}^{\infty}\mu_{k}(\epsilon_{k})\psi_{k}(s)e^{M(\epsilon_{k})(t-s)}dt\big) ds\\
  &=-\frac{\mu_{k}(\epsilon_{k})}{M(\epsilon_{k})}\int_{0}^{\infty}\psi_{k}(s)ds=-\frac{\mu_{k}(\epsilon_{k})}{M(\epsilon_{k})} \mathbb{E}\int_{0}^{\infty}|u_{k}(t)|^{2}dt.
  \end{aligned}\]
  Hence, substituting the above results into \eqref{ZLQ-convexity-concavity}, we have
\[\begin{aligned}
  &J^{0}(0,i;u_{1},0)\!=\!\mathbb{E}\int_{0}^{\infty}\!\left[\left<Q(\alpha)X_{1},X_{1}\right>\!+\!\left<R_{11}(\alpha)u_{1},u_{1}\right>\right]dt\!\geq\! \left(m_{R}\!-\!\frac{\mu_{1}(\epsilon_{1})m_{Q}}{M(\epsilon_{1})}\right)\!
  \mathbb{E}\!\int_{0}^{\infty}|u_{1}|^{2}dt\!\geq\! 0,\\
  &J^{0}(0,i;0,u_{2})\!=\!\mathbb{E}\int_{0}^{\infty}\!\left[\left<Q(\alpha)X_{2},X_{2}\right>\!+\!\left<R_{22}(\alpha)u_{2},u_{2}\right>\right]dt\!\leq \!\left(M_{R}\!-\!\frac{\mu_{2}(\epsilon_{1})M_{Q}}{M(\epsilon_{2})}\right)\!
  \mathbb{E}\!\int_{0}^{\infty}|u_{2}|^{2}dt\!\leq \! 0.
  \end{aligned}\]
  This completes the proof.
\end{proof}

\begin{remark}
    In fact, one can similarly generalize our result to the case
    $m_{Q}<M_{Q}<0$ and $M_{Q}>m_{Q}>0$. Compared with the example constructed in \cite{Sun.J.R.2021_ZSILQ} that corresponding performance functional satisfies the condition \eqref{ZLQ-convexity-concavity} in the one-dimensional case, the sufficient condition pointed above is undoubtedly more general.
\end{remark}

To simplify the further analysis, we set 
\begin{equation}\label{notations-3}
    \begin{array}{l}
    \mathcal{M}(P,i)\triangleq P(i)A(i)+A(i)^{\top}P(i)+C(i)^{\top}P(i)C(i)+Q(i)+\sum_{j=1}^{L}\pi_{ij}P(j),\\
    \mathcal{L}(P,i)\triangleq
    P(i)B(i)+C(i)^{\top}P(i)D(i)+S(i)^{\top},\\
    \mathcal{N}(P,i)\triangleq
    D(i)^{\top}P(i)D(i)+R(i),\quad i\in\mathcal{S},
    \end{array}
\end{equation}
and  consider the following constrained CAREs:
\begin{equation}\label{ZLQ-CAREs}
\left\{
    \begin{aligned}
    &\mathcal{M}(P,i)-\mathcal{L}(P,i) \mathcal{N}(P,i)^{\dag} \mathcal{L}(P,i)^{\top} = 0,\\
    & \mathcal{L}(P,i)\left[I-\mathcal{N}(P,i)\mathcal{N}(P,i)^{\dag}\right]=0.
    \end{aligned}
    \right.
\end{equation}
\begin{definition}\label{def-stabilizing-solution-CAREs}
$\mathbf{P}\in\mathcal{D}\left(\mathbb{S}^{n}\right)$ is called a static stabilizing solution of \eqref{ZLQ-CAREs} if $\mathbf{P}$ solves CAREs \eqref{ZLQ-CAREs},
  and there exists $\mathbf{\Pi}\in\mathcal{D}\left(\mathbb{R}^{m\times n}\right)$ such that
\begin{equation}\label{CAREs-ZLQ-stabilizer}
\mathcal{K}(\mathbf{\Pi})=\left(\mathcal{K}\left(\Pi(1)\right),\mathcal{K}\left(\Pi(2)\right),\cdots,\mathcal{K}\left(\Pi(L)\right)\right)\in\mathcal{H}[A,C;B,D]_{\alpha},
\end{equation}
where $
  \mathcal{K}\left(\Pi(i)\right)\triangleq-\mathcal{N}(P,i)^{\dag}\mathcal{L}(P,i)^{\top}+\left[I-\mathcal{N}(P,i)^{\dag}\mathcal{N}(P,i)\right]\Pi(i),\, i\in\mathcal{S}.
$
\end{definition}
The following result establishes a relation between the optimality system  \eqref{FBSDEs-ZLQ}-\eqref{stationary-ZLQ} and CAREs \eqref{ZLQ-CAREs}.

\begin{theorem}\label{thm-ZLQ-FBSDEs-CAREs}
Suppose that systems $[A,C;B_{1},D_{1}]_{\alpha}$ and $[A,C;B_{2},D_{2}]_{\alpha}$ are $L^{2}$-stablizable. If the following conditions hold:
\begin{description}
\item[(i)] The convexity-concavity condition \eqref{ZLQ-convexity-concavity} holds;
\item[(ii)]The CAREs \eqref{ZLQ-CAREs} admits a static stabilizing solution $\mathbf{P}\in\mathcal{D}\left(\mathbb{S}^n\right)$;
\item[(iii)] The following BSDE
  \begin{equation}\label{eta-ZLQ}
      \begin{aligned}
        d\eta&=-\big\{\big[A(\alpha)^{\top}-\mathcal{L}(P,\alpha)\mathcal{N}(P,\alpha)^{\dag}B(\alpha)^{\top}\big]\eta+\big[C(\alpha)^{\top}-\mathcal{L}(P,\alpha)\mathcal{N}(P,\alpha)^{\dag}D(\alpha)^{\top}\big]\zeta\\
        &\quad+\big[C(\alpha)^{\top}-\mathcal{L}(P,\alpha)\mathcal{N}(P,\alpha)^{\dag}D(\alpha)^{\top}\big]P(\alpha)\sigma-\mathcal{L}(P,\alpha)\mathcal{N}(P,\alpha)^{\dag}\rho+P(\alpha)b+q\big\}dt\\
        &\quad+\zeta dW(t)+\mathbf{z}\cdot d\mathbf{\widetilde{N}}(t),\quad t\geq 0,
      \end{aligned}
  \end{equation}
   admits an adapted solution $(\eta,\zeta,\mathbf{z})\in L_{\mathbb{F}}^{2}(\mathbb{R}^{n}\times L_{\mathbb{F}}^{2}(\mathbb{R}^{n})\times \mathcal{D}\left(L_{\mathcal{P}}^{2}(\mathbb{R}^{n})\right)$ such that
   \begin{equation}\label{ZLQ-eta-constraint}
    \widetilde{\rho}(t)\triangleq B(\alpha_{t})^{\top}\eta(t)+D(\alpha_{t})^{\top}\zeta(t)+D(\alpha_{t})^{\top}P(\alpha_{t})\sigma(t)+\rho(t)\in\mathcal{R}(\mathcal{N}(P,\alpha_{t})),\quad a.e.,\quad a.s..
   \end{equation}
\end{description}
Then, for any $(x,i)\in\mathbb{R}^{n}\times\mathcal{S}$, the optimality system  \eqref{FBSDEs-ZLQ}-\eqref{stationary-ZLQ} admits a solution.  In addition, the Problem (M-ZLQ) is open-loop solvable, and the corresponding closed-loop representation strategy is given by
\begin{equation}\label{ZLQ-closed-loop}
  \left\{
  \begin{aligned}
    \Theta^{*}(i)&=-\mathcal{N}(P,i)^{\dag}\mathcal{L}(P,i)^{\top}+\left[I-\mathcal{N}(P,i)^{\dag}\mathcal{N}(P,i)\right]\Pi(i),\quad\forall i\in\mathcal{S},\\
    \nu^{*}&=-\mathcal{N}(P,\alpha)^{\dag}\widetilde{\rho}
   +\left[I-\mathcal{N}(P,\alpha)^{\dag}\mathcal{N}(P,\alpha)\right]\nu,
  \end{aligned}
  \right.
\end{equation}
where $\mathbf{\Pi}=\left[\Pi(1),\Pi(2),...,\Pi(L)\right]\in\mathcal{D}\left(\mathbb{R}^{m\times n}\right)$ is chosen such that $\mathbf{\Theta}^{*}\in\mathcal{H}[A,C;B,D]_{\alpha}$ and $\nu\in L_{\mathbb{F}}^{2}(\mathbb{R}^{m})$.
\end{theorem}

\begin{proof}

Let $(\mathbf{\Theta}^{*},\nu^{*})$ defined in \eqref{ZLQ-closed-loop} and for any $(x,i)\in\mathbb{R}^{n}\times\mathcal{S}$,
$$
\left\{
\begin{aligned}
&X^{*}\triangleq X(\cdot;x,i,\mathbf{\Theta}^{*},\nu^{*}),\\
&Y^{*}\triangleq P(\alpha)X^{*}(\cdot;x,i)+\eta,\\
&Z^{*}\triangleq P(\alpha)\big[C(\alpha)X^{*}(\cdot;x,i)+D(\alpha)u^{*}(\cdot;x,i)+\sigma\big]+\zeta,\\
&\Gamma_{j}^{*}(t)\triangleq \big[P(j)-P(\alpha(t-))\big]X^{*}(t;x,i)+z_{j}(t),\quad t\geq 0,\quad j\in\mathcal{S},\\
&u^{*}\triangleq \Theta^{*}(\alpha)X^{*}(\cdot;x,i)+\nu^{*},
\end{aligned}
\right.
$$
We claim that $(X^{*},Y^{*},Z^{*},\mathbf{\Gamma}^{*},u^{*})$ defined above solves the optimality system \eqref{FBSDEs-ZLQ}-\eqref{stationary-ZLQ}.
Obviously, by unique solvability, we have  $X^{*}= X(\cdot;x,i,\mathbf{\Theta}^{*},\nu^{*})=X(\cdot;x,i,u^{*})$.
Note that
\begin{equation}
    \begin{aligned}
    \mathcal{N}(P,\alpha)u^{*}&=-\mathcal{N}(P,\alpha)\mathcal{N}(P,\alpha)^{\dag}\mathcal{L}(P,\alpha)^{\top}X^{*}-\mathcal{N}(P,\alpha)\mathcal{N}(P,\alpha)^{\dag}\widetilde{\rho}\\
    &=-\mathcal{L}(P,\alpha)^{\top}X^{*}-\widetilde{\rho}.
    \end{aligned}
\end{equation}
Hence,
\begin{align*}
  &\quad B(\alpha)^{\top}Y^{*}+ D(\alpha)^{\top}Z^{*}+S(\alpha)X^{*}+R(\alpha)u^{*}+\rho\\
  &=B(\alpha)^{\top}(P(\alpha)X^{*}+\eta)+ D(\alpha)^{\top}\big[P(\alpha)\big(C(\alpha)X^{*}+D(\alpha)u^{*}+\sigma\big)+\zeta\big]+S(\alpha)X^{*}+R(\alpha)u^{*}+\rho\\
  &=\mathcal{N}(P,\alpha)u^{*}+\mathcal{L}(P,\alpha)^{\top}X^{*}+\widetilde{\rho}\\
  &=0,
\end{align*}
which proves that the stationary condition holds.

On the other hand, let
\begin{align*}
\Lambda &\triangleq -\big\{\big[A(\alpha)^{\top}-\mathcal{L}(P,\alpha)\mathcal{N}(P,\alpha)^{\dag}B(\alpha)^{\top}\big]\eta+\big[C(\alpha)^{\top}-\mathcal{L}(P,\alpha)\mathcal{N}(P,\alpha)^{\dag}D(\alpha)^{\top}\big]\zeta\\
        &\quad+\big[C(\alpha)^{\top}-\mathcal{L}(P,\alpha)\mathcal{N}(P,\alpha)^{\dag}D(\alpha)^{\top}\big]P(\alpha)\sigma-\mathcal{L}(P,\alpha)\mathcal{N}(P,\alpha)^{\dag}\rho+P(\alpha)b+q\big\}.
\end{align*}
Then applying It\^o's rule to $P(\alpha)X^{*}+\eta$, one has
\begin{align*}
  dY^{*}&=\big\{P(\alpha)\big[A(\alpha)X^{*}+B(\alpha)u^{*}+b\big]+\sum_{j=1}^{L}\pi_{\alpha_{t-},j}P(j)X^{*}+\Lambda\big\}dt+Z^{*}(t)dW(t)+\Gamma^{*}(t)\cdot d\widetilde{N}(t) \\
  &=\big\{\big[\mathcal{L}(P,i) \mathcal{N}(P,i)^{\dag} \mathcal{L}(P,i)^{\top}-A(\alpha)^{\top}P(\alpha)-C(\alpha)^{\top}P(\alpha)C(\alpha)-Q(\alpha)\big]X^{*}+P(\alpha)B(\alpha)u^{*}\\
  &\quad +P(\alpha)b+\Lambda\big\}dt+Z^{*}(t)dW(t)+\Gamma^{*}(t)\cdot d\widetilde{N}(t) \\
  &=\big\{\mathcal{L}(P,i) \mathcal{N}(P,i)^{\dag} \mathcal{L}(P,i)^{\top}X^{*}-A(\alpha)^{\top}\big(Y^{*}-\eta\big)-C(\alpha)^{\top}\big[Z^{*}-P(\alpha)D(\alpha)u^{*}-P(\alpha)\sigma-\zeta\big]\\
 &\quad -Q(\alpha)X^{*}+P(\alpha)B(\alpha)u^{*} +P(\alpha)b+\Lambda\big\}dt+Z^{*}(t)dW(t)+\Gamma^{*}(t)\cdot d\widetilde{N}(t) \\
 &=-\big\{A(\alpha)^{\top}Y^{*}+C(\alpha)^{\top}Z^{*}+Q(\alpha)X^{*}+S(\alpha)^{\top}u^{*}
 -\mathcal{L}(P,i)\big[u^{*}+\mathcal{N}(P,i)^{\dag}\mathcal{L}(P,i)^{\top}X^{*}\big]\\
 &\quad -\big[A(\alpha)^{\top}\eta+C(\alpha)^{\top}\zeta+-C(\alpha)^{\top}P(\alpha)\sigma+P(\alpha)b+\Lambda\big]\big\}dt+Z^{*}(t)dW(t)+\Gamma^{*}(t)\cdot d\widetilde{N}(t) \\
 &=-\big\{A(\alpha)^{\top}Y^{*}+C(\alpha)^{\top}Z^{*}+Q(\alpha)X^{*}+S(\alpha)^{\top}u^{*}+q\\
 &\quad-\mathcal{L}(P,\alpha)\big[u^{*}+\mathcal{N}(P,\alpha)^{\dag}\mathcal{L}(P,\alpha)^{\top}X^{*}+\mathcal{N}(P,\alpha)^{\dag}\widetilde{\rho}\big]\big\}dt+Z^{*}(t)dW(t)+\Gamma^{*}(t)\cdot d\widetilde{N}(t) \\
 &=-\big\{A(\alpha)^{\top}Y^{*}+C(\alpha)^{\top}Z^{*}+Q(\alpha)X^{*}+S(\alpha)^{\top}u^{*}+q\big]\big\}dt+Z^{*}(t)dW(t)+\Gamma^{*}(t)\cdot d\widetilde{N}(t).
\end{align*}
This proves our claim and the desired result follows from Theorem \ref{thm-ZLQ-open-solvability}.
\end{proof}

If  $b=\sigma= q=0$, $\rho_{1}=0$, $\rho_{2}=0$, then one can find from the proof of the above theorem that the BSDE \eqref{eta-ZLQ} admit the unique solution $(\eta,\zeta,\mathbf{z})=(0,0,\mathbf{0})$.
Hence, we have the following result.
\begin{corollary}\label{coro-ZLQ-1}
Suppose that systems $[A,C;B_{1},D_{1}]_{\alpha}$ and $[A,C;B_{2},D_{2}]_{\alpha}$ are $L^{2}$-stablizable. If the following conditions hold:
\begin{description}
\item[(i)]The convexity-concavity condition \eqref{ZLQ-convexity-concavity} holds;
  \item[(ii)] The CAREs \eqref{ZLQ-CAREs} admits a static stabilizing solution $\mathbf{P}\in\mathcal{D}\left(\mathbb{S}^n\right)$;
\end{description}
Then the Problem (M-ZLQ)$^{0}$ is open-loop solvable, and the corresponding closed-loop representation strategy is given by
\begin{equation}\label{ZLQ-closed-loop-0}
  \left\{
  \begin{aligned}
   & \Theta^{*}(i)=-\mathcal{N}(P,i)^{\dag}\mathcal{L}(P,i)^{\top}+\left[I-\mathcal{N}(P,i)^{\dag}\mathcal{N}(P,i)\right]\Pi(i),\quad\forall i\in\mathcal{S},\\
    &\nu^{*}=\left[I-\mathcal{N}(P,\alpha)^{\dag}\mathcal{N}(P,\alpha)\right]\nu,
  \end{aligned}
  \right.
\end{equation}
where $\mathbf{\Pi}=\left[\Pi(1),\Pi(2),...,\Pi(L)\right]\in\mathcal{D}\left(\mathbb{R}^{m\times n}\right)$ is chosen such that $\mathbf{\Theta}^{*}\in\mathcal{H}[A,C;B,D]_{\alpha}$ and $\nu\in L_{\mathbb{F}}^{2}(\mathbb{R}^{m})$.
\end{corollary}

\subsection{Closed-loop solvability for Problem (M-ZLQ)}
The following result provides the closed-loop solvability for Problem (M-ZLQ).
\begin{theorem}\label{thm-ZLQ-closed}
 A 4-tuple $(\mathbf{\widehat{\Theta}_{1}},\widehat{\nu}_{1};\mathbf{\widehat{\Theta}_{2}},\widehat{\nu}_{2})\in\mathcal{D}\left(\mathbb{R}^{m_{1}\times n}\right)\times L_{\mathbb{F}}^{2}(\mathbb{R}^{m_{1}})\times\mathcal{D} \left(\mathbb{R}^{m_{2}\times n}\right)\times L_{\mathbb{F}}^{2}(\mathbb{R}^{m_{2}})$ is a closed-loop Nash equilibrium strategy of Problem (M-ZLQ) if and only if:
 \begin{description}
 \item[(i)] $\mathbf{\widehat{\Theta}}= (\mathbf{\widehat{\Theta}_{1}}^{\top},\mathbf{\widehat{\Theta}_{2}}^{\top})^{\top}\in \mathcal{H}[A,C;B,D]_{\alpha}$,
  \item[(ii)] The  CAREs
  \begin{equation}\label{CAREs-ZLQ}
  0=\mathcal{M}(P,i)-\mathcal{L}(P,i)\mathcal{N}(P_{1},i)^{\dag}\mathcal{L}(P,i)^{\top},
  \end{equation}
 admit a solution $\mathbf{P}\in \mathcal{D}(\mathbb{S}^{n})$ such that
 \begin{equation}\label{CAREs-ZLQ-constraint}
\left\{
\begin{aligned}
&\mathcal{N}(P,i)\widehat{\Theta}(i)+\mathcal{L}(P,i)^{\top}=0,\\
&\mathcal{N}_{11}(P,i)\geq 0,\quad \mathcal{N}_{22}(P,i)\leq 0,\quad \forall i\in\mathcal{S},
\end{aligned}
\right.
\end{equation}
where $\mathcal{N}_{kl}(P,i)\triangleq D_{k}(i)^{\top}P(i)D_{l}(i)+R_{kl}(i)$ for any $ k,l\in\{1,2\}.$
\item[(iii)] The BSDE
\begin{equation}\label{ZLQ-eta}
   \begin{aligned}
   d\eta&=-\big\{\big[A(\alpha)^{\top}-\mathcal{L}(P,\alpha)\mathcal{N}(P,\alpha)^{\dag}B(\alpha)^{\top}\big]\eta+\big[C(\alpha)^{\top}-\mathcal{L}(P,\alpha)\mathcal{N}(P,\alpha)^{\dag}D(\alpha)^{\top}\big]\zeta\\
        &\quad+\big[C(\alpha)^{\top}-\mathcal{L}(P,\alpha)\mathcal{N}(P,\alpha)^{\dag}D(\alpha)^{\top}\big]P(\alpha)\sigma-\mathcal{L}(P,\alpha)\mathcal{N}(P,\alpha)^{\dag}\rho+P(\alpha)b+q\big\}dt\\
        &\quad+\zeta dW(t)+\mathbf{z}\cdot d\mathbf{\widetilde{N}}(t),\quad t\geq 0,
   \end{aligned}
  \end{equation}
 admits a solution $(\eta,\zeta,\mathbf{z}\in L_{\mathbb{F}}^{2}(\mathbb{R}^{n})\times L_{\mathbb{F}}^{2}(\mathbb{R}^{n})\times\mathcal{D}\left(L_{\mathcal{P}}^{2}(\mathbb{R}^{n})\right)$ such that
 \begin{equation}\label{CAREs-ZLQ-eta-constraint}
      \mathcal{N}(P,\alpha_{t})\widehat{\nu}(t)+\widetilde{\rho}(t)=0,\quad a.e. \text{ }a.s.,
  \end{equation}
  where $\widehat{\nu}=(\widehat{\nu}_{1}^{\top},\widehat{\nu}_{2}^{\top})^{\top}$ and $\widetilde{\rho}$  is defined in \eqref{ZLQ-eta-constraint}.
  \end{description}
  In this case, the closed-loop equilibrium value function of Problem (M-ZLQ) is given by
  \begin{equation}\label{ZLQ-value}
  \begin{aligned}
  V(x,i)&=\big<P(i)x,x\big>+\mathbb{E}\big\{2\big<\eta(0),x\big>+\int_{0}^{\infty}
      \big[\big<P(\alpha)\sigma,\sigma\big>+2\big<\eta,b\big>+2\big<\zeta,\sigma\big>-\big<\mathcal{N}(P,\alpha)^{\dag}\widetilde{\rho},\widetilde{\rho}\big>]dt\big\}.
  \end{aligned}
  \end{equation}
\end{theorem}

\begin{proof}
We still consider the Problem (M-ZLQ) as a special case of Problem (M-GLQ) in the sense that the condition \eqref{ZLQ-cost-notation} holds. Then, by Corollary \ref{coro-GLQ-0-closed}, we have
\begin{equation}\label{P1-P2}
  \big<P_{1}(i)x,x\big>=V_{1}^{0}(x,i)=-V_{2}^{0}(x,i)=-\big<P_{2}(i)x,x\big>,\quad \forall (x,i)\in\mathbb{R}^{n}\times\mathcal{S},
\end{equation}
which implies $P_{1}(i)= -P_{2}(i)$ for every $i\in \mathcal{S}$. Without loss of generality, in the rest of this section, we let $P(i)\triangleq P_{1}(i)=-P_{2}(i)$ and define
\begin{equation*}
\mathcal{L}_{k}(P,i)\triangleq P(i)B_{k}(i)+C(i)^{\top}P(i)D_{k}(i)+S_{k}(i)^{\top},
\end{equation*}
Then
$$
\begin{array}{c}
\mathcal{M}(P,i)= \mathcal{M}_{1}(P_{1},i)=-\mathcal{M}_{2}(P_{2},i),\\[2mm]
\mathcal{L}(P,i)=\left(
 \mathcal{L}_{1}(P,i),
 \mathcal{L}_{2}(P,i)
 \right)
 =\left(
 \mathcal{L}_{1}^{1}(P_{1},i),
 \mathcal{L}_{2}^{1}(P_{1},i)
 \right)
 =-\left(
 \mathcal{L}_{1}^{2}(P_{2},i),
 \mathcal{L}_{2}^{2}(P_{2},i)
 \right),\\[2mm]
\mathcal{N}(P,i)\!=\!\left(\!
 \begin{array}{cc}
 \mathcal{N}_{11}(P,i)& \mathcal{N}_{12}(P,i)\\
 \mathcal{N}_{21}(P,i)& \mathcal{N}_{22}(P,i)
 \end{array}
 \!\right)
 \!=\!\left(\!
 \begin{array}{cc}
 \mathcal{N}_{11}^{1}(P_{1},i)& \mathcal{N}_{12}^{1}(P_{1},i)\\
 \mathcal{N}_{21}^{1}(P_{1},i)& \mathcal{N}_{22}^{1}(P_{1},i)
 \end{array}
 \!\right)
  \!=\!-\!\left(\!
 \begin{array}{cc}
 \mathcal{N}_{11}^{2}(P_{2},i)& \mathcal{N}_{12}^{2}(P_{2},i)\\
 \mathcal{N}_{21}^{2}(P_{2},i)& \mathcal{N}_{22}^{2}(P_{2},i)
 \end{array}
 \!\right),
\end{array}
$$
Consequently, the conditions \eqref{CAREs-GLQ-constraint} in this special case is equivalent to
\begin{equation}\label{CAREs-ZLQ-constraint-2}
\left\{
\begin{aligned}
&\mathcal{N}_{11}(P,i)\widehat{\Theta}_{1}(i)+\mathcal{N}_{12}(P,i)\widehat{\Theta}_{2}(i)+\mathcal{L}_{1}(P,i)^{\top}=0,\\
&\mathcal{N}_{21}(P,i)\widehat{\Theta}_{1}(i)+\mathcal{N}_{22}(P,i)\widehat{\Theta}_{2}(i)+\mathcal{L}_{2}(P,i)^{\top}=0,\\
&\mathcal{N}_{11}(P,i)\geq 0,\quad \mathcal{N}_{22}(P,i)\leq 0,
\end{aligned}
\right.
\end{equation}
which in turn is equivalent to \eqref{CAREs-ZLQ-constraint}.
It follows from equations \eqref{CAREs-GLQ-1} and \eqref{CAREs-ZLQ-constraint-2}, we have
\begin{align*}
 0&=\mathcal{M}(P,i)-\widehat{\Theta}_{1}(i)^{\top}\mathcal{N}_{11}(P,i)\widehat{\Theta}_{1}-\widehat{\Theta}_{2}(i)^{\top}\mathcal{N}_{22}(P,i)\widehat{\Theta}_{2}\\
  &\quad+\big[\mathcal{N}_{22}(P,i)\widehat{\Theta}_{2}(i)+\mathcal{L}_{2}(P,i)^{\top}\big]^{\top}\widehat{\Theta}_{2}
  +\widehat{\Theta}_{2}(i)^{\top}\big[\mathcal{N}_{22}(P,i)\widehat{\Theta}_{2}(i)+\mathcal{L}_{2}(P,i)^{\top}\big] \\
  &=\mathcal{M}(P,i)-\widehat{\Theta}(i)^{\top}\mathcal{N}(P,i)\widehat{\Theta}.
  \end{align*}
 Noting that \eqref{CAREs-ZLQ-constraint} implies
$$\widehat{\Theta}(i)^{\top}\mathcal{N}(P,i)\widehat{\Theta}(i)=\mathcal{L}(P,i)\mathcal{N}(P_{1},i)^{\dag}\mathcal{L}(P,i)^{\top}.$$
Combining the above two equations,  we obtain that $\mathbf{P}$ solves CAREs \eqref{CAREs-ZLQ} and such that condition \eqref{CAREs-ZLQ-constraint}.

On the other hand, the condition \eqref{CAREs-GLQ-eta-constraint} in this special case becomes
\begin{equation}\label{CAREs-ZLQ-eta-constraint-2}
  \left\{
      \begin{aligned}
      &\mathcal{N}_{11}(P,\alpha)\widehat{\nu}_{1}+\mathcal{N}_{12}(P,\alpha)\widehat{\nu}_{2}+\bar{\rho}^{1}=0,\\
      &\mathcal{N}_{21}(P,\alpha)\widehat{\nu}_{1}+\mathcal{N}_{22}(P,\alpha)\widehat{\nu}_{2}-\bar{\rho}^{2}=0,
      \end{aligned}
      \right.
  \end{equation}
  where
  \begin{align*}
 &\bar{\rho}^{1}(t)=B_{1}(\alpha_{t})^{\top}\eta_{1}(t)+D_{1}(\alpha_{t})^{\top}\zeta_{1}(t)
+D_{1}(\alpha_{t})^{\top}P(\alpha_{t})\sigma(t)+\rho_{1}(t),\\
&\bar{\rho}^{2}(t)=B_{2}(\alpha_{t})^{\top}\eta_{2}(t)+D_{2}(\alpha_{t})^{\top}\zeta_{2}(t)
-D_{2}(\alpha_{t})^{\top}P(\alpha_{t})\sigma(t)-\rho_{2}(t).
  \end{align*}
Then we can simplify \eqref{GLQ-eta-1}-\eqref{GLQ-eta-2} as
\begin{equation}\label{ZLQ-eta-1}
    \begin{aligned}
    d\eta_{1}&=-\Big\{
   A(\alpha)^{\top}\eta_{1}+C(\alpha)^{\top}\zeta_{1}+C(\alpha)^{\top}P(\alpha)\sigma+P(\alpha)b+q
   +\widehat{\Theta}_{1}(\alpha)^{\top}\bar{\rho}^{1}\\
   &\quad +\widehat{\Theta}_{2}(\alpha)^{\top}\big[B_{2}(\alpha)^{\top}\eta_{1}+D_{2}(\alpha)^{\top}\zeta_{1}
   +D_{2}(\alpha)^{\top}P(\alpha)\sigma+\rho_{2}\big]
   \Big\}dt+\zeta_{1}dW(t)+\mathbf{z}_{1}\cdot d\mathbf{\widetilde{N}}(t),\\
   d\eta_{2}&=-\Big\{
   A(\alpha)^{\top}\eta_{2}+C(\alpha)^{\top}\zeta_{2}-C(\alpha)^{\top}P(\alpha)\sigma-P(\alpha)b-q
   +\widehat{\Theta}_{2}(\alpha)^{\top}\bar{\rho}^{2}\\
   &\quad +\widehat{\Theta}_{1}(\alpha)^{\top}\big[B_{1}(\alpha)^{\top}\eta_{2}+D_{1}(\alpha)^{\top}\zeta_{2}
   -D_{1}(\alpha)^{\top}P(\alpha)\sigma-\rho_{1}\big]
   \Big\}dt+\zeta_{1}dW(t)+\mathbf{z}_{2}\cdot d\mathbf{\widetilde{N}}(t).
    \end{aligned}
\end{equation}
  Adding the above two BSDEs, one can directly derive that
\begin{equation}\label{ZLQ-eta-1+2}
   \begin{aligned}
   d(\eta_{1}+\eta_{2})&=-\big\{\big[A(\alpha)+B(\alpha)\widehat{\Theta}(\alpha)\big]^{\top}(\eta_{1}+\eta_{2})+\big[C(\alpha)+D(\alpha)\widehat{\Theta}(\alpha)\big]^{\top}(\zeta_{1}+\zeta_{2})\big\}dt\\
   &\quad+(\zeta_{1}+\zeta_{2})dW(t)+(\mathbf{z}_{1}+\mathbf{z}_{2})\cdot d\mathbf{\widetilde{N}}(t).
   \end{aligned}
  \end{equation}
By \cite[Proposition $2.7$]{Wu-etal}, the above BSDE admits a unique $L^2$-stable adapted solution,  and we can verify that $(0,0,\mathbf{0})$ is the unique solution of BSDE \eqref{ZLQ-eta-1+2}. Therefore, we have
$$(\eta_{1},\zeta_{1},\mathbf{z}_{1})=-(\eta_{2},\zeta_{2},\mathbf{z}_{2})\triangleq (\eta,\zeta,\mathbf{z}).$$
Consequently, the condition \eqref{CAREs-ZLQ-eta-constraint-2}  is equivalent to \eqref{CAREs-ZLQ-eta-constraint} and \eqref{ZLQ-eta-1} can be rewritten as
\begin{equation*}
   \begin{aligned}
   d\eta&=-\Big\{
   A(\alpha)^{\top}\eta+C(\alpha)^{\top}\zeta+C(\alpha)^{\top}P(\alpha)\sigma+P(\alpha)b+q
   +\widehat{\Theta}(\alpha)^{\top}\mathcal{N}(P,\alpha)\widehat{\nu}\Big\}dt\\
   &\quad+\zeta dW(t)+\mathbf{z}\cdot d\mathbf{\widetilde{N}}(t)\\
   &=-\Big\{
   A(\alpha)^{\top}\eta+C(\alpha)^{\top}\zeta+C(\alpha)^{\top}P(\alpha)\sigma+P(\alpha)b+q
   +\mathcal{L}(P,\alpha)\mathcal{N}(P,\alpha)^{\dag}\widetilde{\rho}\Big\}dt\\
   &\quad+\zeta dW(t)+\mathbf{z}\cdot d\mathbf{\widetilde{N}}(t).
   \end{aligned}
  \end{equation*}
Plugging \eqref{ZLQ-eta-constraint} into the above equation, we can find $(\eta,\zeta,\mathbf{z})$ solves BSDE \eqref{ZLQ-eta}.

Finally, it follows from \eqref{GLQ-value-1} that the  equilibrium value function of Problem (M-ZLQ) satisfies
\begin{equation}\label{ZLQ-value-D}
\begin{aligned}
 V(x,i)&=V_{1}(x,i)=\big<P(i)x,x\big>+\mathbb{E}\Big\{2\big<\eta(0),x\big>+\int_{0}^{\infty}\Big[2\big<\eta,b\big>+2\big<\zeta,\sigma\big>+\big<P(\alpha)\sigma,\sigma\big>\\
  &\quad +2\big<B_2(\alpha)^{\top}\eta+D_2(\alpha)^{\top}\zeta+D_2(\alpha)^{\top}P(\alpha)\sigma+\rho_{2},\widehat{\nu}_{2}\big>+\big<\mathcal{N}_{22}(P,\alpha)\widehat{\nu}_{2},\widehat{\nu}_{2}\big>\\
  &\quad-\big<\mathcal{N}_{11}(P,\alpha)\widehat{\nu}_{1},\widehat{\nu}_{1}\big>\Big]dt\Big\}\\
  &=\big<P(i)x,x\big>+\mathbb{E}\Big\{2\big<\eta(0),x\big>+\int_{0}^{\infty}\Big[2\big<\eta,b\big>+2\big<\zeta,\sigma\big>+\big<P(\alpha)\sigma,\sigma\big>\\
  &\quad -2\big<\mathcal{N}_{21}(P,\alpha)\widehat{\nu}_{1}+\mathcal{N}_{22}(P,\alpha)\widehat{\nu}_{2},\widehat{\nu}_{2}\big>+\big<\mathcal{N}_{22}(P,\alpha)\widehat{\nu}_{2},\widehat{\nu}_{2}\big>-\big<\mathcal{N}_{11}(P,\alpha)\widehat{\nu}_{1},\widehat{\nu}_{1}\big>\Big]dt\Big\}\\
  &=\big<P(i)x,x\big>+\mathbb{E}\big\{2\big<\eta(0),x\big>+\int_{0}^{\infty}\Big[2\big<\eta,b\big>+2\big<\zeta,\sigma\big>+\big<P(\alpha)\sigma,\sigma\big>-\big<\mathcal{N}(P,\alpha)\widehat{\nu},\widehat{\nu}\big>\big]dt\big\}.
\end{aligned}
\end{equation}
Note that
  $$\big<\mathcal{N}(P,\alpha)\widehat{\nu},\widehat{\nu}\big>=\big<\mathcal{N}(P,\alpha)^{\dag}\mathcal{N}(P,\alpha)\widehat{\nu},\mathcal{N}(P,\alpha)\widehat{\nu}\big>=\big<\mathcal{N}(P,\alpha)^{\dag}\widetilde{\rho},\widetilde{\rho}\big>.$$
  The equation \eqref{ZLQ-value-D} implies \eqref{ZLQ-value}. To sum up, we complete the proof.
\end{proof}

Similar to Corollary \ref{coro-GLQ-0-closed}, we have the following result.
\begin{corollary}\label{coro-ZLQ-0-closed}
If $(\mathbf{\widehat{\Theta}_{1}},\widehat{\nu}_{1};\mathbf{\widehat{\Theta}_{2}},\widehat{\nu}_{2})$ is a closed-loop saddle point of Problem (M-ZLQ), then $(\mathbf{\widehat{\Theta}_{1}},0;\mathbf{\widehat{\Theta}_{2}},0)$ is a closed-loop saddle point of Problem (M-ZLQ)$^{0}$. In this case, the closed-loop equilibrium value function of  Problem (M-ZLQ)$^{0}$ is given by
  \begin{equation}\label{ZLQ-0-value}
  V^{0}(x,i)=\big<P(i)x,x\big>.
  \end{equation}
\end{corollary}

  Let $(\mathbf{\widehat{\Theta}},\widehat{\nu})$ be a closed-loop Nash equilibrium strategy of Problem (M-ZLQ), $\mathbf{P}$ be the solution to CAREs \eqref{CAREs-ZLQ} such that condition \eqref{CAREs-ZLQ-constraint} and $(\eta,\zeta,\mathbf{z})$ be the solution to BSDE \eqref{ZLQ-eta} such that constraint \eqref{CAREs-ZLQ-eta-constraint}. Then, by the basic properties of pseudo-inverse, one can easily verify that $\mathbf{P}$ is a static stabilizing solution of CAREs \eqref{ZLQ-CAREs} and $(\eta,\zeta,\mathbf{z})$ also solves BSDE \eqref{eta-ZLQ} with constraint
\eqref{ZLQ-eta-constraint}. Hence, we have the following result from Theorem \ref{thm-ZLQ-FBSDEs-CAREs}.

\begin{corollary}\label{coro-ZLQ-2}
Suppose that systems $[A,C;B_{1},D_{1}]_{\alpha}$ and $[A,C;B_{2},D_{2}]_{\alpha}$ are $L^{2}$-stablizable. If the following conditions hold:
\begin{description}
\item[(i)]The convexity-concavity condition \eqref{ZLQ-convexity-concavity} holds;
  \item[(ii)] The Problem (ZLQ) is closed-loop solvable with closed-loop saddle point $(\mathbf{\widehat{\Theta}},\widehat{\nu})$.
\end{description}
Then the Problem (M-ZLQ) is open-loop solvable and $(\mathbf{\widehat{\Theta}},\widehat{\nu})$ also is a closed-loop representation strategy.
\end{corollary}

\section{Examples}\label{section-examples}
This section presents three concrete examples to illustrate the results in previous sections. The numerical algorithms used here are similar to those in \cite{Jianhui-Huang-2015,Li-Zhou-Rami-2003-ID-MLQ-IF}.
Without loss of generality, we suppose the state space of the Markov chain $\alpha$ is $\mathcal{S}:=\left\{1,2,3\right\}$ and the corresponding generator is given by
$$\Pi=\left[\begin{array}{ccc}
     \pi_{11}   &   \pi_{12}   &   \pi_{13}\\
    \pi_{21}   &   \pi_{22}   &   \pi_{23}\\
    \pi_{31}   &   \pi_{32}   &   \pi_{33}
\end{array}\right]
=\left[\begin{array}{ccc}
     -0.5   &   0.3   &   0.2\\
    0.2   &   -0.4   &   0.2\\
    0.3   &   0.2   &   -0.5
\end{array}\right].$$
In addition, the state process in this section satisfies:
\begin{equation}\label{state-exam}
  \left\{
 \begin{aligned}
   dX(t)&=\left[A\left(\alpha_{t}\right)X(t)+B_{1}\left(\alpha_{t}\right)u_{1}(t)+B_{2}\left(\alpha_{t}\right)u_{2}(t)\right]dt\\
   &\quad+\left[C\left(\alpha_{t}\right)X(t)+D_{1}\left(\alpha_{t}\right)u_{1}(t)+D_{2}\left(\alpha_{t}\right)u_{2}(t)\right]dt,\\
   X(0)&=x,\quad \alpha_{0}=i.
   \end{aligned}
  \right.
\end{equation}
For the first two examples, the coefficients of the state equation are set to:
\begin{align*}
&A(1)=\left(\begin{matrix}-3&1\\0&-5 \end{matrix}\right),\,
A(2)=\left(\begin{matrix}-4&0\\0&-3 \end{matrix}\right),\,
A(3)=\left(\begin{matrix}-5&1\\-1&-4 \end{matrix}\right),\,
C(1)=\left(\begin{matrix}1&1\\0&-1 \end{matrix}\right),\\
&C(2)=\left(\begin{matrix}1&-1\\1&0\end{matrix}\right),\,
C(3)=\left(\begin{matrix}0&1\\-1&1 \end{matrix}\right),\,
B_{1}(1)=\left(\begin{matrix}-2&1\\3&-5 \end{matrix}\right),\,
B_{1}(2)=\left(\begin{matrix}-1&0\\3&-3 \end{matrix}\right),\\
&B_{1}(3)=\left(\begin{matrix}-2&1\\0&-4 \end{matrix}\right),\,
B_{2}(1)=\left(\begin{matrix}1&1\\2&-5 \end{matrix}\right),\,
B_{2}(2)=\left(\begin{matrix}3&0\\1&-3 \end{matrix}\right),\,
B_{2}(3)=\left(\begin{matrix}1&1\\-1&-4 \end{matrix}\right),\\
&D_{1}(1)=\left(\begin{matrix}2&1\\0&-1 \end{matrix}\right),\,
D_{1}(2)=\left(\begin{matrix}-1&0\\2&-3 \end{matrix}\right),\,
D_{1}(3)=\left(\begin{matrix}-4&0\\-1&-3 \end{matrix}\right),\,
D_{2}(1)=\left(\begin{matrix}-3&2\\1&-5 \end{matrix}\right),\\
&D_{2}(2)=\left(\begin{matrix}-1&1\\1&-3 \end{matrix}\right),\,
D_{2}(3)=\left(\begin{matrix}-3&-1\\-1&0 \end{matrix}\right).
\end{align*}
However, for the last example, except for the coefficients $A(i),\,C(i)$, $i\in\mathcal{S}$ remaining unchanged, the other coefficients are changed to:
\begin{equation}\label{state-coefficients-5-3}
\begin{aligned}
&B_{1}(1)=\left(\begin{matrix}-1&0\\ 1& 1 \end{matrix}\right),\,
B_{1}(2)=\left(\begin{matrix}-1& 1\\ 1& 0 \end{matrix}\right),\,
B_{1}(3)=\left(\begin{matrix}0& 1\\ 1& 2 \end{matrix}\right),\,
B_{2}(1)=\left(\begin{matrix}1& 1\\ 0&-1 \end{matrix}\right),\\
&B_{2}(2)=\left(\begin{matrix}2& 0\\ 1& -1 \end{matrix}\right),\,
B_{2}(3)=\left(\begin{matrix}0&1\\2&-1 \end{matrix}\right),\,
D_{1}(1)=\left(\begin{matrix}1&0\\ 1&-1 \end{matrix}\right),\,
D_{1}(2)=\left(\begin{matrix}-1&0\\1&1 \end{matrix}\right),\\
&D_{1}(3)=\left(\begin{matrix}-1& 0\\ -1& -1 \end{matrix}\right),\,
D_{2}(1)=\left(\begin{matrix}1& 0\\1& 1 \end{matrix}\right),\,
D_{2}(2)=\left(\begin{matrix}0& 1\\1& 0 \end{matrix}\right),\,
D_{2}(3)=\left(\begin{matrix}1& -1\\0&1 \end{matrix}\right).
\end{aligned}
\end{equation}
In the above, we can easily verify that $A(i)+A(i)^{\top}+C(i)^{\top}C(i)<0$ for any $i\in\mathcal{S}$. It follows from Remark $2.4$ in \cite{Wu-etal} that system $[A,C]_{\alpha}$ is $L^{2}$-stable.

The following example considers the open-loop solvable and closed-loop solvable of Problem (M-GLQ)$^{0}$. Both closed-loop representations of open-loop Nash equilibrium point and closed-loop Nash equilibrium strategy are obtained. It also shows that the solution to CAREs \eqref{decouple-P-3}-\eqref{decouple-P-constraint-3} is asymmetric while the solution to CAREs \eqref{CAREs-GLQ-1}-\eqref{CAREs-GLQ-constraint} is symmetric from the numerical point of view.

\begin{example}\rm
For $k=1,2$, consider the following cost functionals of Problem (M-GLQ)$^{0}$:
 \begin{equation}\label{cost-GLQ-exam}
 \begin{aligned}
     J_{k}\left(x,i;u_{1},u_{2}\right)
     & = \mathbb{E}\int_{0}^{\infty}
     \left<
     \left(
     \begin{matrix}
     Q^{k}(\alpha_{t})&S_{1}^{k}(\alpha_{t})^{\top}&S_{2}^{k}(\alpha_{t})^{\top}\\
     S_{1}^{k}(\alpha_{t})&R_{11}^{k}(\alpha_{t})&R_{12}^{k}(\alpha_{t})\\
     S_{2}^{k}(\alpha_{t})&R_{21}^{k}(\alpha_{t})&R_{22}^{k}(\alpha_{t})
     \end{matrix}
     \right)
     \left(
     \begin{matrix}
     X(t)\\
     u_{1}(t)\\
     u_{2}(t)
     \end{matrix}
     \right),
     \left(
     \begin{matrix}
     X(t)\\
     u_{1}(t)\\
     u_{2}(t)
     \end{matrix}
     \right)
     \right>dt,
   \end{aligned}
 \end{equation}
 where
\begin{align*}
    &\left(
    \begin{matrix}
    Q^{1}(1)        &    S_{1}^{1}(1)^{\top}    &      S_{2}^{1}(1)^{\top}\\
    S_{1}^{1}(1)    &    R_{11}^{1}(1)          &      R_{12}^{1}(1)\\
    S_{2}^{1}(1)    &    R_{21}^{1}(1)          &      R_{22}^{1}(1)
    \end{matrix}
    \right)=
    \left(
    \begin{matrix}
    1.67  &  -0.05  &  -0.02  &  -0.04   &  -1.25  &   0.13\\
   -0.05  &   1.63  &   0.01  &   0.06   &  -0.11  &   0.03\\
   -0.02  &   0.01  &   1.61  &  -0.13   &   0.01  &  -1.03\\
   -0.04  &   0.06  &  -0.13  &   1.69   &  -1.04  &   0.18\\
   -1.25  &  -0.11  &   0.01  &  -1.04   &  -0.13  &   0.04\\
    0.13  &   0.03  &  -1.03  &   0.18   &   0.04  &  -0.15
    \end{matrix}
    \right),\\
    &\left(
    \begin{matrix}
    Q^{1}(2)        &    S_{1}^{1}(2)^{\top}      &      S_{2}^{1}(2)^{\top}\\
    S_{1}^{1}(2)    &    R_{11}^{1}(2)            &      R_{12}^{1}(2)\\
    S_{2}^{1}(2)    &    R_{21}^{1}(2)            &      R_{22}^{1}(2)
    \end{matrix}
    \right)=
    \left(
    \begin{matrix}
    1.70  &   0.08  &   0.02  &   0.00   &  -1.15  &   0.13\\
    0.08  &   1.66  &   0.07  &   0.01   &  -0.11  &   0.06\\
    0.02  &   0.07  &   1.59  &  -0.10   &   0.07  &  -1.03\\
    0.00  &   0.01  &  -0.10  &   1.65   &  -1.04  &   0.08\\
   -1.15  &  -0.11  &   0.07  &  -1.04   &  -0.19  &   0.02\\
    0.13  &   0.06  &  -1.03  &   0.08   &   0.02  &  -0.13
    \end{matrix}
    \right),\\
    &\left(
    \begin{matrix}
    Q^{1}(3)        &    S_{1}^{1}(3)^{\top}    &    S_{2}^{1}(3)^{\top}\\
    S_{1}^{1}(3)    &    R_{11}^{1}(3)          &    R_{12}^{1}(3)\\
    S_{2}^{1}(3)    &    R_{21}^{1}(3)          &    R_{22}^{1}(3)
    \end{matrix}
    \right)=
    \left(
    \begin{matrix}
    1.59  &  -0.04  &  -0.03  &   0.06   &  -1.15  &   0.03\\
   -0.04  &   1.77  &  -0.01  &   0.06   &   0.11  &   0.06\\
   -0.03  &  -0.01  &   1.57  &   0.05   &   0.07  &  -0.03\\
    0.06  &   0.06  &   0.05  &   1.67   &  -0.04  &   0.08\\
   -1.15  &   0.11  &   0.07  &  -0.04   &  -0.11  &  -0.02\\
    0.03  &   0.06  &  -0.03  &   0.08   &  -0.02  &  -0.15
    \end{matrix}
    \right),\\
      &\left(
    \begin{matrix}
    Q^{2}(1)        &    S_{1}^{2}(1)^{\top}    &      S_{2}^{2}(1)^{\top}\\
    S_{1}^{2}(1)    &    R_{11}^{2}(1)          &      R_{12}^{2}(1)\\
    S_{2}^{2}(1)    &    R_{21}^{2}(1)          &      R_{22}^{2}(1)
    \end{matrix}
    \right)=
    \left(
    \begin{matrix}
     1.69  &   -0.02  &   -0.02  &    0.00   &    0.09  &   -0.04\\
    -0.02  &    1.52  &   -0.07  &   -0.01   &   -0.05  &   -0.03\\
    -0.02  &   -0.07  &   -1.59  &    0.10   &   -0.07  &    1.03\\
     0.00  &   -0.01  &    0.10  &   -1.65   &    1.04  &   -0.08\\
     0.09  &   -0.05  &   -0.07  &    1.04   &    1.68  &   -0.03\\
    -0.04  &   -0.03  &    1.03  &   -0.08   &   -0.03  &    1.71
    \end{matrix}
    \right),\\
    &\left(
    \begin{matrix}
    Q^{2}(2)        &    S_{1}^{2}(2)^{\top}      &      S_{2}^{2}(2)^{\top}\\
    S_{1}^{2}(2)    &    R_{11}^{2}(2)            &      R_{12}^{2}(2)\\
    S_{2}^{2}(2)    &    R_{21}^{2}(2)            &      R_{22}^{2}(2)
    \end{matrix}
    \right)=
    \left(
    \begin{matrix}
     1.67  &    0.05  &    0.02  &    0.04   &    0.14  &   -0.03\\
     0.05  &    1.64  &   -0.01  &   -0.06   &   -0.01  &    0.05\\
     0.02  &   -0.01  &   -1.61  &    0.13   &   -0.01  &    1.03\\
     0.04  &   -0.06  &    0.13  &   -1.69   &    1.04  &   -0.18\\
     0.14  &   -0.01  &   -0.01  &    1.04   &    1.65  &   -0.01\\
    -0.03  &    0.05  &    1.03  &   -0.18   &   -0.01  &    1.65
    \end{matrix}
    \right),\\
    &\left(
    \begin{matrix}
    Q^{2}(3)        &    S_{1}^{2}(3)^{\top}    &    S_{2}^{2}(3)^{\top}\\
    S_{1}^{2}(3)    &    R_{11}^{2}(3)          &    R_{12}^{2}(3)\\
    S_{2}^{2}(3)    &    R_{21}^{2}(3)          &    R_{22}^{2}(3)
    \end{matrix}
    \right)=
    \left(
    \begin{matrix}
   1.59    &    0.03    &    0.03    &   -0.06     &   -0.08    &   -0.07\\
   0.03    &    1.61    &    0.01    &   -0.06     &    0.05    &   -0.01\\
   0.03    &    0.01    &   -1.57    &   -0.05     &   -0.07    &    0.03\\
  -0.06    &   -0.06    &   -0.05    &   -1.67     &    0.04    &   -0.08\\
  -0.08    &    0.05    &   -0.07    &    0.04     &    1.75    &   -0.05\\
  -0.07    &   -0.01    &    0.03    &   -0.08     &   -0.05    &    1.62
    \end{matrix}
    \right).
\end{align*}
We point out that the above-given coefficients are required to satisfy condition \eqref{GLQ-convexity-condition}. Hence, the above cost functionals satisfy the convexity condition \eqref{GLQ-convexity}.
\end{example}

To obtain the closed-loop representation of the open-loop Nash equilibrium point, one must solve the constrained CAREs
\eqref{decouple-P-3}-\eqref{decouple-P-constraint-3}. Using the SDP technique, we obtain the numerical approximation solution as follows:
\begin{align*}
 &P_{1}(1)\!=\!\left[\begin{matrix}
 0.299917 &  0.062075\\
0.053329 &  0.173778
 \end{matrix}
 \right],\,
 P_{1}(2)\!=\!\left[\begin{matrix}
 0.284044 & -0.023310\\
-0.018440  & 0.246700
 \end{matrix}
 \right],\,
 P_{1}(3)\!=\!\left[\begin{matrix}
 0.169793  & 0.003297\\
0.006381  & 0.169369
 \end{matrix}
 \right],\\
 &P_{2}(1)\!=\!\left[\begin{matrix}
 0.317295 &  0.064732\\
0.060875 &  0.166318
 \end{matrix}
 \right],\,
 P_{2}(2)\!=\!\left[\begin{matrix}
 0.214041 & -0.027260\\
-0.018476 &  0.239175
 \end{matrix}
 \right],\,
 P_{2}(3)\!=\!\left[\begin{matrix}
 0.161026  & 0.010229\\
0.016542 &  0.151165
 \end{matrix}
 \right],
\end{align*}
One can verify that the Frobenius norm error of the above numerical solutions can be precise to $10^{-6}$. Substituting the numerical solutions into the following equation
$$
\Sigma(\mathbb{P},i)=\left(\begin{matrix}
R_{11}^{1}(i)+D_{1}(i)^{\top}P_{1}(i)D_{1}(i) & R_{12}^{1}(i)+D_{1}(i)^{\top}P_{1}(i)D_{2}(i)\\
R_{21}^{2}(i)+D_{2}(i)^{\top}P_{2}(i)D_{1}(i) & R_{22}^{2}(i)+D_{2}(i)^{\top}P_{2}(i)D_{2}(i)
\end{matrix}\right),\quad i\in\mathcal{S},
$$
we can find that, for any $i\in\mathcal{S}$, the $\Sigma(\mathbb{P},i)$ is invertible. Solving $\mathbf{\Theta^{*}}$ from \eqref{decouple-P-constraint-3}, we can obtain
\begin{align*}
\mathbf{\Theta^{*}}&=\left(\Theta^{*}(1),\Theta^{*}(2),\Theta^{*}(3)\right)\\
&=\left(
\left[\begin{matrix}
  -0.005455 \!&\! -0.357973\\
  -0.034157 \!&\!  0.209280\\
   0.073937 \!&\! -0.035403\\
  -0.022786 \!&\!  0.018593
\end{matrix}\right]\!,\!
\left[\begin{matrix}
   0.217417 \!&\! -0.204680\\
   0.032269 \!&\! -0.004510\\
  -0.440579 \!&\! -0.012847\\
   0.021868 \!&\!  0.188713
\end{matrix}\right]\!,\!
\left[\begin{matrix}
   0.135889 \!&\!  0.051144\\
  -0.230434 \!&\!  0.320481\\
  -0.109368 \!&\!  0.086127\\
  -0.051071 \!&\!  0.403312
\end{matrix}\right]\right).
\end{align*}
It is easy to verify that
\begin{equation}\label{exam-AC}
A(i)+B(i)\Theta^{*}(i)+(A(i)+B(i)\Theta^{*}(i))^{\top}+(C(i)+D(i)\Theta^{*}(i))^{\top}(C(i)+D(i)\Theta^{*}(i))<0, \quad i\in\mathcal{S}.
\end{equation}
Hence, we have $\mathbf{\Theta^{*}}\in\mathcal{H}[A,C;B,D]_{\alpha}$.
It follows from Corollary \ref{coro-GLQ-open-solvability-0} that the above example is open-loop solvable, and any open-loop Nash equilibrium point admits the following closed-loop representation:
$$u^{*}(\cdot;x,i)=(u_{1}^{*}(\cdot;x,i)^{\top},u_{2}^{*}(\cdot;x,i)^{\top})^{\top}
=\Theta^{*}(\alpha)X^{0}(\cdot;x,i,\mathbf{\Theta^{*}},0).$$

Next, we consider the closed-loop equilibrium point of this example. To this end, we need to solve the constrained CAREs \eqref{CAREs-GLQ-1}-\eqref{CAREs-GLQ-constraint}, or equivalently, CAREs\eqref{CAREs-GLQ-1-2}-\eqref{CAREs-GLQ-constraint-2}. Using the SDP technique again, we obtain the numerical approximation solution as follows:
\begin{align*}
 &P_{1}(1)\!=\!\left[\begin{matrix}
 0.284341 &  0.053644\\
 0.053644 &  0.177155
 \end{matrix}
 \right],\,
 P_{1}(2)\!=\!\left[\begin{matrix}
  0.299351 & -0.021964\\
 -0.021964 &  0.248061
 \end{matrix}
 \right],\,
 P_{1}(3)\!=\!\left[\begin{matrix}
0.176490 &  0.005881\\
0.005881 &  0.156280
 \end{matrix}
 \right],\\
 &P_{2}(1)\!=\!\left[\begin{matrix}
  0.315689 &  0.065654\\
  0.065654 &  0.119917
 \end{matrix}
 \right],\,
 P_{2}(2)\!=\!\left[\begin{matrix}
  0.199450 & -0.004419\\
 -0.004419 &  0.214877
 \end{matrix}
 \right],\,
 P_{2}(3)\!=\!\left[\begin{matrix}
 0.146503 &  0.026456\\
 0.026456 &  0.127204
 \end{matrix}
 \right],
\end{align*}
We claim that the Frobenius norm error of the above numerical solutions also can be precise to $10^{-6}$. Substituting the above numerical solutions into \eqref{CAREs-GLQ-constraint}, we obtain
\begin{align*}
    \mathbf{\widehat{\Theta}}&=\left(
    \left[\begin{matrix}\widehat{\Theta}_{1}(1)\\\widehat{\Theta}_{2}(1)\end{matrix}\right],
    \left[\begin{matrix}\widehat{\Theta}_{1}(2)\\\widehat{\Theta}_{2}(2)\end{matrix}\right],
    \left[\begin{matrix}\widehat{\Theta}_{1}(3)\\\widehat{\Theta}_{2}(3)\end{matrix}\right]
    \right)\\
    &=\left(
    \left[\begin{matrix}
  -0.013230 & -0.358866\\
  -0.021533 &  0.215132\\
   0.068180 & -0.023482\\
  -0.025340 &  0.042498
\end{matrix}\right],
\left[\begin{matrix}
   0.246859 & -0.212545\\
   0.037244 & -0.018403\\
  -0.429070 & -0.045493\\
   0.034625 &  0.183211
\end{matrix}\right],
\left[\begin{matrix}
    0.130566 &  0.073648\\
   -0.229734 &  0.304168\\
   -0.103527 &  0.070543\\
   -0.021569 &  0.333343
\end{matrix}\right]
    \right).
\end{align*}
We can also verify that the above solution satisfies condition \eqref{exam-AC}, which implies $\mathbf{\widehat{\Theta}}\in\mathcal{H}[A,C;B,D]_{\alpha}$. It follows from Corollary \ref{coro-GLQ-0-closed} that $(\mathbf{\widehat{\Theta}},0)$ is the closed-loop Nash equilibrium strategy of this example.

The next example studies the closed-loop solvable of Problem (M-ZLQ)$^{0}$.
\begin{example}\label{example-2}\rm
Consider the following cost functionals of Problem (M-ZLQ)$^{0}$:
 \begin{equation}\label{cost-ZLQ-exam}
     J\left(x,i;u_{1},u_{2}\right) \!=\! \mathbb{E}\int_{0}^{\infty}\!\left[\left<Q(\alpha_{t})X(t),X(t)\right>\!+\!\left<R_{11}(\alpha_{t})u_{1}(t),u_{1}(t)\right>\!+\!\left<R_{22}(\alpha_{t})u_{2}(t),u_{2}(t)\right>\right]dt,
 \end{equation}
 where
 \begin{align*}
&Q(1)=\left(\begin{matrix}1.11&0.11\\ 0.11& 1.02 \end{matrix}\right),\,
Q(2)=\left(\begin{matrix}1.01& 0.13\\ 0.13& -1.12 \end{matrix}\right),\,
Q(3)=\left(\begin{matrix}-1.01& 0.21\\ 0.21& -1.02\end{matrix}\right),\\
&R_{11}(1)=\left(\begin{matrix}3.11& 1.03\\ 1.03&2.19 \end{matrix}\right),\,
R_{11}(2)=\left(\begin{matrix}5.37& -0.48\\ -0.48& 5.63 \end{matrix}\right),\,
R_{11}(3)=\left(\begin{matrix}5.11&0.23\\0.23 & 6.19 \end{matrix}\right),\\
&R_{22}(1)=\left(\begin{matrix}-6.82&-0.34\\ -0.34&-5.88 \end{matrix}\right),\,
R_{22}(2)=\left(\begin{matrix}-6.01&-0.14\\-0.14&-5.88 \end{matrix}\right),\,
R_{22}(3)=\left(\begin{matrix}-3.82&-1.04\\-1.04&-2.88 \end{matrix}\right).
\end{align*}
\end{example}
Solving constrained CAREs \eqref{CAREs-ZLQ}-\eqref{CAREs-ZLQ-constraint}, we obtain
\begin{align*}
 &P(1)\!=\!\left[\begin{matrix}
 0.201466   &   0.054055\\
 0.054055   &   0.102922
 \end{matrix}
 \right],\,
 P(2)\!=\!\left[\begin{matrix}
   0.091779   &    -0.005053\\
  -0.005053   &   -0.206651
 \end{matrix}
 \right],\,
 P(3)\!=\!\left[\begin{matrix}
  -0.156160   &     0.045988\\
   0.045988    &   -0.166436
 \end{matrix}
 \right].
\end{align*}
The accuracy of the above numerical solution can still reach $10^{-6}$. Plugging the above solutions into \eqref{CAREs-ZLQ-constraint}, we have
\begin{align*}
    \mathbf{\widehat{\Theta}}&=\left(
    \left[\begin{matrix}\widehat{\Theta}_{1}(1)\\\widehat{\Theta}_{2}(1)\end{matrix}\right],
    \left[\begin{matrix}\widehat{\Theta}_{1}(2)\\\widehat{\Theta}_{2}(2)\end{matrix}\right],
    \left[\begin{matrix}\widehat{\Theta}_{1}(3)\\\widehat{\Theta}_{2}(3)\end{matrix}\right]
    \right)\\
    &=\left(
    \left[\begin{matrix}
  -0.047278 & -0.209312\\
  -0.020200 &  0.209932\\
  -0.039428 &-0.030322\\
   0.023878  & 0.047724
\end{matrix}\right],
\left[\begin{matrix}
   0.116822 &  0.101032\\
  -0.123268 & -0.121514\\
  -0.014035 & -0.030890\\
   0.143139  & 0.110170
\end{matrix}\right],
\left[\begin{matrix}
    -0.139384  & -0.045542\\
     0.189717    & -0.231162\\
     0.018796    &   0.083455\\
    -0.070742   &   0.228681
\end{matrix}\right]
    \right),
\end{align*}
which can be verified to satisfy the condition \eqref{exam-AC}. According to Theorem \ref{thm-ZLQ-closed},  $(\mathbf{\widehat{\Theta}},0)$ is a closed-loop saddle point of this example.

In the above example, the minimum eigenvalue $m_{Q}$ and maximum eigenvalue $M_{Q}$ of process $Q(\alpha)$ obviously satisfy $m_{Q}<0,\,M_{Q}>0$. However, the condition (ii) in Proposition \ref{prop-convex-concave} does not hold. Hence, we can not demonstrate that the above example is open-loop solvable and construct a closed-loop representation strategy. As we can see in the next example, whose coefficients completely satisfied the assumptions proposed in Proposition \ref{prop-convex-concave}. Hence, we can further investigate its open-loop solvability and closed-loop solvability.

\begin{example}\label{example-3}\rm
Still considering the Problem (M-ZLQ)$^{0}$, whose performance functional is consistent with \eqref{cost-ZLQ-exam}. Compared with Example \ref{example-2}, we just update the value of $B_{1}(i),\, B_{2}(i),\, D_{1}(i)$, $D_{2}(i)$, $i\in\mathcal{S}$, to \eqref{state-coefficients-5-3} and modify the value of $R_{11}(1)$ and  $R_{22}(3)$ to
\[
R_{11}(1)=\left(\begin{matrix}4.11& 1.03\\ 1.03&6.19 \end{matrix}\right),\quad R_{22}(3)=\left(\begin{matrix}-5.82&-1.04\\-1.04&-5.88 \end{matrix}\right).
\]
\end{example}
Let $\epsilon_{1}=\epsilon_{2}=1$. We can verify that all assumptions proposed in Proposition \ref{prop-convex-concave} are satisfied. Hence, the performance functional of this example satisfies the condition \eqref{ZLQ-convexity-concavity}. To obtain the closed-loop representation strategy of open-loop saddle point, we need to solve CAREs \eqref{ZLQ-CAREs} and obtain the numerical solutions within the error level $10^{-6}$ as follows
\begin{align*}
 &P(1)\!=\!\left[\begin{matrix}
 0.208881 &  0.056594\\
0.056594 &  0.119522
 \end{matrix}
 \right],\,
 P(2)\!=\!\left[\begin{matrix}
   0.097670 & 0.000771\\
   0.000771 & -0.194399
 \end{matrix}
 \right],\,
 P(3)\!=\!\left[\begin{matrix}
   -0.135584 &  0.036854\\
   0.036854 & -0.160611
 \end{matrix}
 \right].
\end{align*}
Substituting the above numerical solutions into \eqref{ZLQ-closed-loop}, we have
\begin{align*}
\mathbf{\Theta^{*}}&=\left(\Theta^{*}(1),\Theta^{*}(2),\Theta^{*}(3)\right)\\
&=\left(
\left[\begin{matrix}
  -0.034391 \!&\!  -0.030087\\
   0.007260 \!&\! -0.024662\\
   0.070939 \!&\!  0.022068\\
   0.033041 \!&\! -0.022884
\end{matrix}\right]\!,\!
\left[\begin{matrix}
    0.077430 \!&\!  0.017903\\
   0.027153 \!&\!  0.001070\\
  -0.003238 \!&\! -0.032191\\
   0.015675 \!&\!  0.017204
\end{matrix}\right]\!,\!
\left[\begin{matrix}
   0.016852 \!&\! -0.011859\\
   0.036492 \!&\!  0.025103\\
   0.005594 \!&\! -0.076541\\
   0.004448 \!&\!  0.038961
\end{matrix}\right]\right),
\end{align*}
which also satisfies the condition \eqref{exam-AC}. In addition, one can verify that $\mathcal{N}(P,i)$ is invertible, which implies that
$$\nu^{*}=\left[I-\mathcal{N}(P,\alpha)^{\dag}\mathcal{N}(P,\alpha)\right]\nu=0.$$
It follows from Corollary \ref{coro-ZLQ-1} that $(\mathbf{\Theta^{*}},0)$ is the closed-loop representation strategy of this example.

On the other hand, we point that $\mathbf{P}=[P(1),P(2),P(3)]$ satisfies the following condition:
$$\mathcal{N}_{11}(P,i)\geq 0,\quad \mathcal{N}_{22}(P,i)\leq 0,\quad \forall i\in\mathcal{S}.$$
Hence, $\mathbf{P}$ also solves the constrained CAREs \eqref{CAREs-ZLQ}-\eqref{CAREs-ZLQ-constraint} and
the closed-loop representation strategy $(\mathbf{\Theta^{*}},0)$ also is the closed-loop saddle point.

\end{document}